\documentclass[11pt,A4paper]{article}

\usepackage[left=30mm,right=30mm,top=25mm,bottom=30mm]{geometry}
\usepackage{labelfig}
\usepackage{epsfig}
\usepackage{epstopdf}
\usepackage{soul}
\usepackage{xfrac}

\usepackage[percent]{overpic}

\usepackage{color}
\usepackage{amsthm,amsmath,amssymb}
\usepackage{booktabs}
\usepackage{mathpazo}
\usepackage{microtype}
\usepackage{overpic}
\usepackage{bm}
\usepackage{sectsty}
\usepackage[
	pdftitle={PDFTitle},
	pdfauthor={Hugo Parlier},
	ocgcolorlinks,
	linkcolor=linkred,
	citecolor=linkred,
	urlcolor=linkblue]
{hyperref}

\definecolor{linkred}{RGB}{0,191,255} 
\definecolor{linkblue}{RGB}{16, 78, 139}

\usepackage[hang,flushmargin]{footmisc}
\usepackage{enumitem}
\usepackage{titlesec}
	\titlespacing{\section}{0pt}{12pt}{0pt}
	\titlespacing{\subsection}{0pt}{6pt}{0pt}
	
\titlelabel{\thetitle.\quad}

\makeatletter 

\long\def\@footnotetext#1{%
\H@@footnotetext{%
\ifHy@nesting 
\hyper@@anchor{\@currentHref}{#1}%
\else 
\Hy@raisedlink{\hyper@@anchor{\@currentHref}{\relax}}#1%
\fi 
}}

\def\@footnotemark{%
\leavevmode 
\ifhmode\edef\@x@sf{\the\spacefactor}\nobreak\fi 
\H@refstepcounter{Hfootnote}%
\hyper@makecurrent{Hfootnote}%
\hyper@linkstart{link}{\@currentHref}%
\@makefnmark 
\hyper@linkend 
\ifhmode\spacefactor\@x@sf\fi 
\relax 
}%

\makeatother 

\theoremstyle{plain}
\newtheorem{theorem}{Theorem}[section]
\newtheorem*{theorem-otal}{Theorem 1.3}
\newtheorem{proposition}[theorem]{Proposition}
\newtheorem{lemma}[theorem]{Lemma}
\newtheorem{corollary}[theorem]{Corollary}

\newtheorem{claim}[theorem]{Claim}

\theoremstyle{definition}

\newtheorem{remark}[theorem]{Remark}

\newcommand{\R}{{\mathbb R}}

\newcommand{\N}{{\mathbb N}}

\newcommand{\Z}{{\mathbb Z}}
\newcommand{\C}{{\mathcal C}}

\newcommand{\ZZ}{\mathbb{Z}}

\newcommand{\Homeo}{\mathrm{Homeo}}
\newcommand{\MCG}{\mathrm{Mod}}

\sectionfont{\large \bfseries}
\subsectionfont{\normalsize}

\long\def\symbolfootnote[#1]#2{\begingroup%
\def\thefootnote{\fnsymbol{footnote}}\footnote[#1]{#2}\endgroup}

\def\blfootnote{\xdef\@thefnmark{}\@footnotetext}

\begin{document}

{\Large \bfseries A topological viewpoint on curves via intersection}

{\large Hugo Parlier
and Binbin Xu\symbolfootnote[1]{\small Both authors were supported by the Luxembourg National Research Fund OPEN grant O19/13865598.\\
Binbin Xu was also supported by the Fundamental Research Funds for the Central Universities, Nankai University (Grant Number 63231055), Natural Science Foundation of Tianjin (Grant number 22JCYBJC00690) and LPMC at Nankai University.
}}

\vspace{0.5cm}

{\bf Abstract.}
This paper explores the relationship between closed curves on surfaces and their intersections. Like Dehn-Thurston coordinates for simple curves, we explore how to determine closed curves using the number of times they intersect other curves. We construct and study $k$-equivalent curves: these are distinct curves that intersect all curves with $k$ self-intersection points the same number of times. We show that such curves must intersect all simple curves in the same way, but that all other possible implications fail. Our methods give a quantitative approach to a theorem of Otal which shows that curves are determined by their intersection with all other curves. In the opposite direction, we show that non-simple curves can only be distinguished by looking at their intersection with infinitely many curves.

\vspace{0.5cm}

\section{Introduction} \label{s:introduction}

The study of closed curves on surfaces has proved to be a fascinating topic, and has been explored from algebraic, geometric, topological, dynamical and combinatorial viewpoints. Curves have been vital in understanding surfaces, their moduli spaces, and related structures. As a starting point, simple curves allow one to break surfaces into elementary pieces (such as pairs of pants or polygons), and thus became tools for the study of spaces such as Teichm\"uller spaces which encode hyperbolic (or conformal) structures. From the point of view of a dynamical system, when studying the geodesic flow on a (say hyperbolic) surface, periodic orbits play a vital role. The lengths of these orbits give the length spectrum, and this in turn leads to the spectral theory of surfaces.

From a topological point of view, the study of intersection between curves was already an essential ingredient in the work of Dehn \cite{Dehn}, precursive in its algorithmic nature. More generally, via the work of Thurston and others, simple curves have played a privileged role in geometric topology. Thurston reprised Dehn's approach to use intersection numbers to parametrize the set of isotopy classes of simple closed curves (and more generally measured laminations) using so-called Dehn-Thurston coordinates. With these coordinates established, it only requires a small step to see that there exists a reference finite collection of curves such that the intersection numbers with these curves determine unique isotopy classes of simple closed curves. For closed curves, a similar type of coordinate system does not exist, and one of the main goals of this paper is to understand to which extent similar notions can be generalized in this setting.

A clue that things become more complicated in the general framework comes from the study of length equivalent curves. These are pairs of curves such that the corresponding geodesics have the same length for any hyperbolic metric. These curves, exhibited by Randol using results of Horowitz \cite{Horowitz,Randol}, have a remarkable intersection property: despite being non-isotopic they intersect all simple closed curves the same number of times. An additional surprise came with the work of Leininger \cite{Leininger} who showed the existence of curves that are {\it not} length equivalent despite sharing this same intersection property. To further confuse matters, Chas \cite{Chas2014} exhibited examples of length equivalent curves with different self-intersections. This opened the door to understanding $0$-equivalence for curves: under which conditions do a pair of curves intersect every simple closed geodesic the same number of times?

The set of closed curves are naturally stratified by self-intersection number. This is the point of view in the multiple recent generalizations \cite{EPS, Erlandsson-Souto, Erlandsson-SoutoBook, Sapir} of Mirzakhani's simple curve counting theorem \cite{Mirzakhani}. These theorems illustrate how the set of curves with $k$-self-intersections behave in a similar fashion to simple closed curves. Erlandsson and Souto's generalization of train-tracks to so-called {\it radallas} is a concrete example of this. In this spirit, Leininger's notion of $0$-equivalence can be promoted to $k$-equivalence: two non-isotopic curves are $k$-equivalent if they intersect all curves with self-intersection $k$ the same number of times. To the best of our knowledge, the definition of $k$-equivalent curves was first formulated by Kavi \cite{Kavi} for curves on a pair of pants.

Our first result is about this notion. Throughout the paper, curves will lie on $\Sigma$, a closed orientable surface of genus $g\geq 2$.

\begin{theorem}\label{thm:fromkto0}
For any integer $k>0$, a pair of $k$-equivalent curves are also $0$-equivalent.
\end{theorem}

Of course, for the statement to have content, we have to show that $k$-equivalent curves exist. Furthermore, as conjectured by Kavi for a pair of pants, one might wonder whether $k$-equivalence implies $(k-1)$-equivalence, but in fact this is far from being the case in our setting.

\begin{theorem}
Let $k\neq k'$ be strictly positive integers. Then there exist a pair of curves that are $k$-equivalent but not $k'$-equivalent. 
\end{theorem}

More generally, given disjoint finite collections $K, K'$ of positive integers, we can find a pair of curves that are $k$-equivalent for any $k\in K$ but not $k'$-equivalent for any $k' \in K'$.

In light of this, one might think that you could find a pair of curves that can never be distinguished by intersection, but this is not the case. The following result is a slight variation on a theorem of Otal \cite{Otal}:

\begin{theorem}\label{thm:otal}
Any pair of distinct curves are not $k$-equivalent for infinitely many $k$.
\end{theorem}

The original result of Otal is that a geodesic current, as defined by Bonahon \cite{Bonahon}, is determined by its intersection with all curves. A curve is a type of current, so Otal's result implies that any two curves that have the same intersection with any other curve are freely homotopic. In the opposite direction, the above result, using the density of curves inside the space of currents, can be used to obtain Otal's result in its entirety. Our proof and approach provide the following quantitative version of the theorem above:

\begin{theorem}
A pair of distinct curves $\alpha$ and $\beta$ with self-intersection smaller or equal to $k\geq 0$ are not $k'$-equivalent for some $k'\le17k$.
\end{theorem}

This relates to a question we first raised: what does a reference system for all closed curves - that is a collection of curves whose intersection determines a curve - look like? Provided you have information on the self-intersection of the curves to be detected, the above result tells you that you need only look at the intersection with curves with bounded self-intersection. However, our methods also show that, unlike Dehn-Thurston coordinates for simple curves, a coordinate system necessarily consists of infinitely many curves. More precisely, if we denote by $\C_k$ the set of homotopy classes of curves with $k$-self-intersections, we prove:

\begin{corollary}\label{cor:infinite}
Given any finite collection of curves $\eta_1,\hdots,\eta_n$, and an integer $k>0$, the map $
\varphi: \C_k \to \N^n
$
given by
$$
\varphi(\alpha)= \left(i(\alpha,\eta_1), \hdots, i(\alpha,\eta_n)\right)
$$
is not injective.
\end{corollary}

Perhaps the most surprising part of this is that knowing the self-intersection of the curve,unless it is $0$, does not help. This is due to the following phenomena: 
\begin{enumerate}
	\item A curve with self-intersections can be always obtained by considering a simple multicurve and a transversal pants decomposition, and then by adding crossings by cutting and pasting segments that pass through a given pants curve. This is the statement of Lemma \ref{lem:pants}, one of our main technical results which we call the {\it pants lemma}. 
	\item It requires finitely many curves to detect simple multicurves, but infinitely many curves to detect the crossings. This is due to the fact that any curve used to detect the crossing necessarily changes as one applies different elements of the mapping class group. Hence, one needs not just a single curve, but its full mapping class group orbit. 
\end{enumerate}

We end this introduction with a few words about how our understanding evolved while working on this project. Many of these questions came from the study of geodesics on hyperbolic surfaces, including the study of length equivalent curves. In fact, our first proofs of the existence of $k$-equivalent curves used the seminal Birman-Series result \cite{Birman-Series85} on the sparseness of geodesics with bounded intersection. Ultimately however, our statements are topological and so we wanted a purely topological proofs and constructions. At one point, we began to realize that finding sets of reference curves for non-simple curves was what at the heart of what we were after. The "pants lemma" (Lemma \ref{lem:pants}) was a milestone for us and allowed us to "rid" ourselves of any non-topological techniques. It also allowed us to reprove Otal's theorem with elementary arguments and lead us to the fact that to distinguish between non-simple curves, you need to check intersection with infinitely many curves.

\medskip

{\bf Organization.}

This paper is organized as follows. We begin by background and definitions. Then, in Section \ref{sec:kto0}, we use the so-called cylinder lemma to show that $k$-equivalence implies $0$-equivalence. In Section \ref{sec:pants}, we prove the pants lemma and use it to prove that distinct curves are not $k$-equivalent for infinitely many $k$. Section \ref{sec:tobeornottobe} is dedicated to the construction of $k$-equivalent curves which are not $k'$ equivalent and related results. The main result of Section \ref{sec:infinite} is that any reference system consists infinitely many curves.

{\bf Acknowledgments.}

We thank Hanh Vo for bringing Kavi's paper to our attention.

\section{Preliminaries}

\subsection{Curves on surfaces}
Let $\Sigma$ be a closed orientable surface of genus $g\geq 2$. A \textit{curve} on $\Sigma$ is the image of a continuous map from the circle $S^1$ to $\Sigma$. We say that two curves $\gamma_1$ and $\gamma_2$ are \textit{homotopic} to each other if there is a homotopy of $\Sigma$ sending $\gamma_1$ to $\gamma_2$, and we write $\gamma_1\sim\gamma_2$. A curve is said to be \textit{essential} if it is not homotopic to a point on $\Sigma$ (the image of a constant map of $S^1$). A curve is said to be \textit{simple} if it is essential and homotopic to a curve without self-intersections, or in other words, the image of an embedding of $S^1$ to $\Sigma$. Throughout, we only consider primitive curves, meaning that they are not proper powers of another curve and all curves will be essential. For simplicity, we will use the word curve to denote both a homotopy class of a curve and its representatives, and the meaning should be clear depending on the context. 

For $\gamma_1$ and $\gamma_2$ two essential curves, $i(\gamma_1,\gamma_2)$ denotes minimal geometric intersection between their representatives. Unless explicitly stated, we assume that given any collection of curves, they are in a pairwise minimal intersecting position, and thus each intersection point is transverse. This leads to the notion of a taut representation (explained in more detail below for multicurves).

Note that with this convention, if $\gamma_1=\gamma_2=\gamma$, then there exists a positive integer $k$, such that $i(\gamma,\gamma)=2k$ and we call $\gamma$ a \textit{$k$-curve}. We denote by $\C_k$ the collection of $k$-curves so for example $\C_0$ denotes the set of simple curves. Two distinct curves $\alpha$ and $\beta$ are said to be \textit{$k$-equivalent} if for any $\gamma\in\C_k$, we have $i(\alpha,\gamma)=i(\beta,\gamma)$. We denote $k$-equivalence by $\alpha\sim_k\beta$. 

It would be difficult to write a paper about curves and intersections without at least mentioning the classical problem of determining self-intersection of a curve (or a collection of curves) via some set of data. Curves correspond to elements of $\pi_1(\Sigma)$, so a generating set of $\pi_1(\Sigma)$ gives a natural description of the set of curves as conjugacy classes of words. How this description relates to intersection numbers is exactly at the heart of the different algorithms which aim to compute the self-intersection number of a curve \cite{Birman-Series87, Cohen-Lustig, Tan, Despre-Lazarus}. Further relations between words and intersection have been explored from different points of view \cite{Chas2015, Chas-Lalley}. 

\subsection{Multicurves and tautness}
A {\it multicurve} is a collection of curves. If the curves are simple and disjoint, we have a simple multicurve. A pants decomposition is a type of simple multicurve, given by a maximal collection of distinct and pairwise disjoint simple curves.

The representation of a multicurve is said to be {\it taut} if it is in minimal intersecting position. The notion of taut is trickier than one might think at first. One way of obtaining a taut representation is by endowing $\Sigma$ with a hyperbolic metric and looking at geodesic representatives of curves, but this does not exactly work for multicurves or arcs. A curve with self-intersections can also have multiple taut representations, even up to ambient homeomorphism isotopic to the identity. These different representations can be obtained one from another by a sequence of elementary moves, not unlike Reidemeister moves for knots, while maintaining tautness \cite{Schrijver}. A way to check whether a multicurve is in taut position is by using what is often referred to as the bigon criterion, that a lack of (immersed) bigons or monogons is sufficient to conclude that the curves are in minimal position \cite{Hass-Scott}. 

Let $M=\{\gamma_1,\dots,\gamma_n\}$ and $M'=\{\gamma_1',\dots,\gamma_{n'}'\}$ be multicurves. Their intersection number is defined as the minimal number of points of intersection between representatives, and is equal to the the sum of the intersections of the curves that comprise them:
	\[
		i(M,M')=\sum_{i=1}^n\sum_{j=1}^{n'}i(\gamma_i,\gamma_j').
	\]
\subsection{Subsurfaces and arcs}
A \textit{subsurface} of $\Sigma$ is a connected subset of $\Sigma$ whose boundary is a collection of pairwise disjoint simple curves on $\Sigma$. Let $Y$ be a subsurface of $\Sigma$ and $\eta_1,\dots,\eta_m$ the connected components of its boundary. If there are distinct indices $i$ and $j$, such that $\eta_i$ and $\eta_j$ are homotopic to each other, then $Y$ is homeomorphic to a cylinder and we have $\partial Y=\eta_i\cup\eta_j$. A curve is homotopic to $\eta_i$ (hence homotopic to $\eta_j$ as well) is then called a \textit{core curve} of $Y$.

Conversely, by cutting $\Sigma$ along a collection of distinct and disjoint simple curves, we get a collection of subsurfaces. On $\Sigma$, a closed surface of genus $g\geq 2$, a pants decomposition contains $3g-3$ curves and cuts $\Sigma$ into $2g-2$ subsurfaces, each of which is a pair of pants (homeomorphic to a thrice punctured sphere).

Let $Y$ be a subsurface of $\Sigma$. An \textit{arc} in $Y$ is the image of a continuous map from $[0,1]$ to $Y$, such that the image of $0$ and $1$ are both on $\partial Y$.

We will adopt the following convention: if $Y$ is a cylinder, we consider homotopies of arcs with fixed endpoints, but if $Y$ has more topology, we allow free homotopies with endpoints that can move along the boundary. As for curves, if two arcs $a$ and $b$ are homotopic to each other, we write $a\sim b$. An arc in $Y$ is said to be \textit{essential} if it is not homotopic to a segment on $\partial Y$. An arc is said to be \textit{simple} if it is essential and homotopic to an arc which has no self-intersection. As for curves, we only consider essential arcs and we use word arc to denote both a homotopy class of an arc and its representatives.

Note that in order to avoid using the term arc in too many ways, we will use the term {\it segment} for a a subarc of an arc or a curve. In particular, endpoints of a segment can lie on the boundary or be interior points.

Intersection between arcs $a$ and $b$ is defined as for curves, that is minimal among representatives, and is denoted $i(a,b)$. An arc that has no self-intersections is called \textit{simple}. A collection of arcs will be referred to as a multiarc, and if the arcs are all simple and pairwise disjoint, it is a simple multiarc. Note that the notion of taut naturally extends to representations of collections of multiarcs and multicurves.

\subsection{Mapping class groups of surfaces}
Let $\Homeo^+(\Sigma)$ be the group of orientation preserving homeomorphisms of $\Sigma$ and $\Homeo^+_0(\Sigma)$ its normal subgroup consisting of homeomorphisms homotopic to identity map. The \textit{mapping class group} of $\Sigma$ is defined to be
\[
\MCG(\Sigma)= \Homeo^+(\Sigma)/\Homeo^+_0(\Sigma).
\]
We will use the term mapping class for elements in $\MCG(\Sigma)$ and their representatives.

The mapping class group can be generated using Dehn twists which will play an important part in our approach. Let $\eta$ be a simple curve and $U$ be a cylindric neighborhood of $\eta$. The \textit{left Dehn twist} along $\eta$, denoted by $D_\eta$, is a homeomorphism of $\Sigma$ which is the identity map on $\Sigma\setminus U$ and whose restriction on $U$ is a full twist to the left (see Figure \ref{fig:dtwist})

\begin{figure}[h]
\begin{center}
\includegraphics[width=5cm]{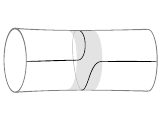}
\vspace{-24pt}
\end{center}
\caption{A left Dehn twist}
\label{fig:dtwist}
\end{figure}



The \textit{right Dehn twist} is its inverse and is denoted $D_\eta^{-1}$.


The mapping class group $\MCG(\Sigma)$ acts on the set of curves, and this action preserves the intersection numbers between curves. In particular, the sets $\C_k$ are left globally invariant for any $k\geq 0$.

For $Y$ a non-cylindrical subsurface of $\Sigma$, let $\MCG(\Sigma,Y)$ be the subgroup of $\MCG(\Sigma)$ preserving $Y$, and that preserves each connected component of $\partial Y$ globally. If $Y$ is a cylindrical subsurface, then $\MCG(\Sigma,Y)$ is the subgroup of $\MCG(\Sigma)$ that preserves the boundary of $Y$ pointwise. In both cases, $\MCG(\Sigma,Y)$ acts on the set of arcs in $Y$. Moreover, this action preserves the intersection numbers between arcs.

Note that the situation on a pair of pants is very different, as the mapping class group is finite and hence there are only finitely many curves with bounded intersection. Kavi \cite{Kavi} shows that being $k$-equivalent implies being $1$ and $2$-equivalent, but this only involves checking the intersection with a finite number of curves.

\section{From $k$-equivalence to $0$-equivalence}\label{sec:kto0}

This part is dedicated to the proof of Theorem \ref{thm:fromkto0} which states that two curves that are $k$-equivalent are also $0$-equivalent. To do so we begin by investigating intersections of arcs in a cylinder.

\subsection{Intersection among arcs in a cylinder}
Let $U\subset \Sigma$ be a cylindric neighborhood of a simple curve $\eta$. Given two curves $\alpha$ and $\beta$, we are interested in the asymptotic growth of the intersection numbers $i(D^s_\eta(\alpha),\beta)$ for $s\in\ZZ$ as $s$ goes to infinity.

We begin by studying an elementary model. Let $b$ and $r$ be two arcs in $U$ with different (fixed) endpoints on $\partial U$. Note the arcs are necessarily simple. 

The cylinder has two boundary components (a left and a right boundary component). For convenience, we give the arcs an orientation from the left boundary to the right boundary. 
With an orientation of the surface (the counter-clockwise direction is positive), this allows us to talk about positive and negative intersection points of ordered and oriented arcs: considering the sign of the rotation from the direction of the first arc (say $b$) to that of the second one (say $r$), to each intersection point one associates either $+$ or $-$.

Let us first consider the case where $b$ and $r$ intersect each other once. There are two possible configurations depending on whether their intersection is positive or negative (See Figure \ref{fig:dtwist-intersection}).

Observe the following about the effects of Dehn twists along $\eta$:
\begin{itemize}
	\item If $b$ intersects $r$ negatively, we have
	\[
	\begin{aligned}
	i(D_\eta({\color{blue}b}),{\color{red}r})&=i({\color{blue}b},{\color{red}r})+1,&&\textrm{(top left)};\\
	i(D_\eta^{-1}({\color{blue}b}),{\color{red}r})&=i({\color{blue}b},{\color{red}r})-1,&&\textrm{(top right)};	\end{aligned}
	\]
	\item If $b$ intersects $r$ positively, we have
	\[
	\begin{aligned}
	i(D_\eta({\color{blue}b}),{\color{red}r})&=i({\color{blue}b},{\color{red}r})-1,&&\textrm{(bottom left)};\\
	i(D_\eta^{-1}({\color{blue}b}),{\color{red}r})&=i({\color{blue}b},{\color{red}r})+1,&&\textrm{(bottom right)}.	
	\end{aligned}
	\]
\end{itemize}
\begin{figure}[h]
	\begin{center}
		\includegraphics[scale=1]{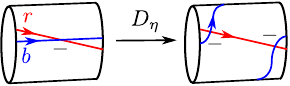}
		\includegraphics[scale=1]{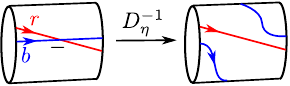}
	\end{center}
	\begin{center}
		\includegraphics[scale=1]{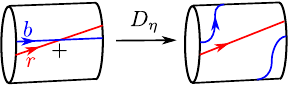}
		\includegraphics[scale=1]{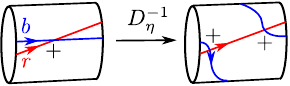}
	\end{center}
	\caption{Dehn twists and intersection numbers in a cylinder}
	\label{fig:dtwist-intersection}
\end{figure}

Notice that if $b$ intersects $r$ multiple times, all intersections will have the same sign. This is because otherwise there will be a bigon formed by a segment of $b$ and a segment of $r$. Therefore it makes sense to say that they intersect positively or negatively. Also note that, in the case of multiple intersections, when we apply a Dehn twist along $\eta$ on $b$, by restricting to a "thinner" subcylinder of $U$ containing only one intersection point between $b$ and $r$, the twist happens on this subcylinder. And so the local picture on this subcylinder is as above. We summarize the result of this argument as follows:

\begin{lemma}\label{lem:arcs}
	If $b$ intersects $r$ negatively, we have
	\[
	\begin{aligned}
	i(D_\eta({\color{blue}b}),{\color{red}r})&=i({\color{blue}b},{\color{red}r})+1;\\
	i(D_\eta^{-1}({\color{blue}b}),{\color{red}r})&=i({\color{blue}b},{\color{red}r})-1;	
	\end{aligned}
	\]
	If $b$ intersects $r$ positively, we have
	\[
	\begin{aligned}
	i(D_\eta({\color{blue}b}),{\color{red}r})&=i({\color{blue}b},{\color{red}r})-1;\\
	i(D_\eta^{-1}({\color{blue}b}),{\color{red}r})&=i({\color{blue}b},{\color{red}r})+1.	
	\end{aligned}
	\]	
 If $b$ and $r$ are disjoint, we have
	\[
	\begin{aligned}
	i(D_\eta({\color{blue}b}),{\color{red}r})&=1,&&\textrm{sgn}(D_\eta({\color{blue}b})\cap{\color{red}r})=-;\\
	i(D_\eta^{-1}({\color{blue}b}),{\color{red}r})&=1, &&\textrm{sgn}(D_\eta^{-1}({\color{blue}b})\cap{\color{red}r})=+.
	\end{aligned}
	\]
\end{lemma}
We can now pass to the main technical tool of this section.

\subsection{The cylinder lemma}

Let $\eta$ be a simple closed curve, and $\alpha$ a non-simple closed curve.
Consider representatives of $\alpha$ and $\eta$ such that they intersect minimally and such that the set $Q$ of self-intersection points of $\alpha$ is disjoint from $\eta$. This means that there is a neighborhood $V_0$ of $\eta$ homeomorphic to a cylinder and disjoint from $Q$ (see Figure \ref{fig:cylinderneighborhood}).

\begin{figure}[h]
\leavevmode \SetLabels
\L(.48*.16) $V_0$\\%
\L(.48*.85) $\eta$\\%
\L(.20*.26) $\alpha$\\%
\endSetLabels
\begin{center}
\AffixLabels{\centerline{\includegraphics[width=10cm]{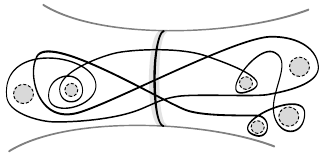}}}
\vspace{-24pt}
\end{center}
\caption{The cylinder neighborhood $V_0$}
\label{fig:cylinderneighborhood}
\end{figure}

By cutting $\Sigma$ along $\eta$, we obtain a (possibly disconnected) surface with two boundary components. By our assumption on $\alpha$, $Q$ and $\eta$, the curve $\alpha$ is cut into a collection of essential arcs $\{a_1,\dots,a_m\}$ in $\Sigma\setminus\eta$, and $Q$ is contained in $\Sigma\setminus\eta$. Moreover, each point in $Q$ is either an intersection point between two distinct arcs $a_i$ and $a_j$ or a self-intersection of an arc $a_i$. Notice that, at this moment, the intersections between arcs in $\{a_1,\dots,a_m\}$ are not necessarily minimal (even up to homotopy with respect to boundary).

In the case where $q\in Q$ lies in the intersection of two arcs belonging to $\{a_1,\dots,a_m\}$, up to possible relabelling, we assume that $q\in a_1\cap a_2$. We denote by $a_1^+$ and $a_1^-$ (resp. $a_2^+$ and $a_2^-$) the two endpoints of $a_1$ (resp. $a_2$) on the boundary of $\Sigma\setminus\eta$. Starting from one endpoint of $a_1$, we follow $a_1$ until $q$, then we continue along $a_2$ back to the boundary. In this way, we obtain an arc in $\Sigma\setminus\eta$. By considering all possible combinations, we obtain $4$ different arcs connecting the pairs of endpoints $\{a_1^+,a_2^+\}$, $\{a_1^+,a_2^-\}$, $\{a_1^-,a_2^+\}$ and $\{a_1^-,a_2^-\}$. We denote these respectively by $b_1$, $b_2$, $b_3$ and $b_4$ (see Figure \ref{fig:endpoints}). 

\begin{figure}[h]
\leavevmode \SetLabels
\L(.378*.9) $a_1^+$\\%
\L(.383*.65) $a_2^+$\\%
\L(.39*.4) $a_1^-$\\%
\L(.395*.13) $a_2^-$\\%
\L(.62*.57) $b$\\%
\L(.672*.71) $a_1^+$\\%
\L(.68*.46) $a_1^-$\\%
\L(.30*.70) $b_1$\\%
\L(.32*.46) $b_3$\\%
\L(.30*.30) $b_4$\\%
\L(.30*.10) $b_2$\\%
\L(.26*.50) $q$\\%
\L(.565*.53) $q$\\%
\endSetLabels
\begin{center}
\AffixLabels{\centerline{\includegraphics[width=8cm]{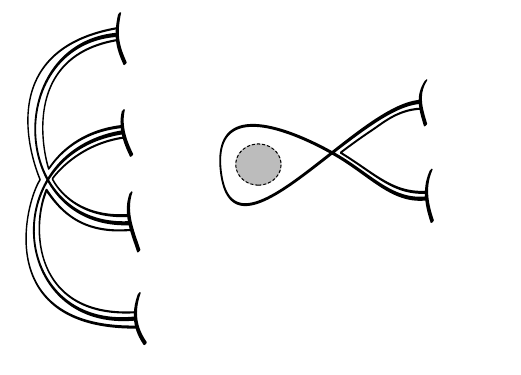}}}
\vspace{-24pt}
\end{center}
\caption{The arcs and their endpoints}
\label{fig:endpoints}
\end{figure}
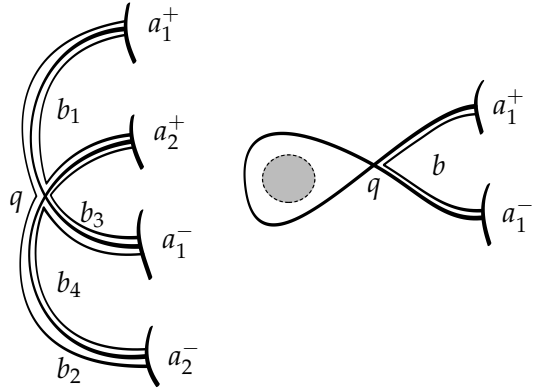

When $q\in Q$ is a self-intersection point of an arc in $\{a_1,\dots,a_m\}$, up to relabelling, we assume it is $a_1$ and denote by $a_1^+$ and $a_1^-$ its endpoints. From each endpoint, we have a segment of $a_1$ connecting it to $q$. By combining them we get an arc in $\Sigma\setminus\eta$ which we denote by $b$ (see Figure \ref{fig:endpoints}). 

A point $q$ in $Q$ is said to be \textit{$\eta$-admissible} if, up to relabeling the arcs $\{a_1,\dots,a_m\}$, one of the following holds:
\begin{enumerate}
	\item $q$ is in the intersection of $a_1$ and $a_2$, and one of the arcs $b_1$, $b_2$, $b_3$ and $b_4$ is non-essential;
	\item $q$ is a self-intersection of $a_1$ and the arc $b$ is non-essential.
\end{enumerate}

In what follows, whenever $q\in Q$ is $\eta$-admissible, we denote by $b_q$ the (or one of the) non-essential arc formed by two segments in $\alpha$ connecting $q$ to $\eta$. If moreover $\eta$ has an orientation, we say that $q$ is $\eta$-admissible \textit{from the left} (resp. right) if these two segments intersect $\eta$ from its left side (resp. its right side).

\begin{remark}\label{etak}
It is useful to observe that an arc in a surface with boundary starting and ending at a same boundary component is non-essential if and only if, by connecting its endpoints using a boundary segment, we obtain a closed curve either homotopically trivial or homotopic to a power of the boundary curve.
\end{remark}

We now describe a process where we change representatives of $\alpha$ by explicit homotopies. Throughout, the representatives of the multicurve $\eta \cup \alpha$ remain taut and we continue to denote the points of self-intersection of $\alpha$ by $Q$. Our main goal is to show the following:

\begin{lemma}[The cylinder lemma]\label{cylinderlemma}
The curve $\alpha$ admits a taut representative such that there exists a cylindric neighborhood $U$ for $\eta$ satisfying the following conditions.
\begin{enumerate}
	\item $\partial U\cap Q=\emptyset$
	\item Any $q\in Q\cap U$ is also an intersection point between two arcs of $\alpha\cap U$.
\item $U$ contains all $\eta$-admissible points of $Q$. 
\end{enumerate}
\end{lemma}
\begin{proof}[Proof of Lemma \ref{cylinderlemma}]

If $\alpha$ and $\eta$ are disjoint, then $\eta$ admits a cylindric neighborhood disjoint from $\alpha$, hence the lemma holds vacuously. 
	
When $\alpha$ and $\eta$ intersect, the proof is an iterative construction to get a curve $\alpha'$ homotopic to $\alpha$, which forms no monogon or bigon with itself or $\eta$. Hence, by the bigon criterion, the multicurve $\alpha'\cup\eta$ is taut. 
	
Let $q\in Q$ be $\eta$-admissible. Let $b_q$ be the associated non-essential arc constructed as above. By cutting $b_q$ at $q$, we get two segments (denoted $e$ and $e'$) oriented from $q$ to $\eta$. Let $p$ and $p'$ be the endpoints of $e$ and $e'$ on $\eta$. 
	
It will be convenient to consider $\eta$ with a fixed orientation. Without loss of generality, we assume that $q$ is $\eta$-admissible from the left, that is, both $e$ and $e'$ intersect $\eta$ from the left side. The points $p$ and $p'$ separate $\eta$ into two segments, and we denote by $c$ the one from $p$ to $p'$ following the orientation of $\eta$. By hypothesis, the loop $e'\ast (e\ast c)^{-1}$ is homotopic to $\eta^k$ for some $k\in\Z$ (see Remark \ref{etak}). 

We first consider the case when $k=0$. Consider
	\[
		e"=e'\ast(e\ast c)^{-1}\ast(e\ast c)\sim e'.
	\]
Since $e'\ast(e\ast c)^{-1}$ is homotopic to a point, we conclude that both $e'$ and $e"$ are homotopic to $e\ast c$. 

\begin{figure}[h!]
\leavevmode \SetLabels
\L(.37*.91) $e'$\\%
\L(.37*.08) $e$\\%
\L(.465*.47) $c$\\%
\L(.24*.52) $q$\\%
\L(.505*.52) $q$\\%
\L(.62*.57) $e''$\\%
\L(.465*.85) $p'$\\%
\L(.465*.2) $p$\\%
\L(.72*.85) $p'$\\%
\L(.72*.2) $p$\\%
\endSetLabels
\begin{center}
\AffixLabels{\centerline{\includegraphics[width=8cm]{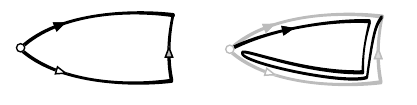}}}
\vspace{-24pt}
\end{center}
\caption{Replacing $e'$}
\label{fig:replace}
\end{figure}
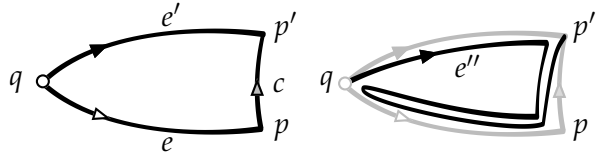

Therefore, we can replace $e'$ in $\alpha$ by $e\ast c$ to get a new representative of $\alpha$.
	
We now consider the case where $k\neq 0$. Without loss of generality, we assume that $k>0$, and
	\[
		e'\ast(e\ast c)^{-1}\sim\eta^k.
	\]

We consider a cylinder neighborhood $U$ of $\eta$, such that $U\cap Q=\emptyset$. Hence $\alpha\cap U$ is a collection of disjoint simple arcs. The curve $\eta$ separates $U$ into two cylinders, and we denote by $U^+$ the one on the left. The boundary $\eta^+$ of $U^+$ which is different from $\eta$ separates $e$ (resp. $e'$) into two segments $e_1$ and $e_2$ (resp. $e_1'$ and $e_2'$). We denote by $p_+$ (resp. $p_+'$) the separating point. The orientation on $\eta$ induces an orientation on $\eta^+$. The segment from $p_+$ to $p_+'$ following this orientation is denoted by $c_+$.

\begin{figure}[h]
\leavevmode \SetLabels
\L(.03*.52) $q$\\%
\L(.245*.575) $p_+'$\\%
\L(.245*.31) $p_+$\\%
\L(.325*.575) $p_m'$\\%
\L(.33*.35) $p_m$\\%
\L(.375*.575) $p'$\\%
\L(.375*.39) $p$\\%
\L(.62*.425) $e_1$\\%
\L(.61*.63) $e_1'$\\%
\L(.84*.34) $e_3$\\%
\L(.84*.57) $e_3'$\\%
\L(.885*.36) $e_4$\\%
\L(.885*.57) $e_4'$\\%
\L(.798*.445) $c_+$\\%
\L(.838*.45) $c_m$\\%
\L(.90*.46) $c$\\%
\L(.30*.83) $\eta_+$\\%
\L(.35*.83) $\eta_m$\\%
\L(.40*.83) $\eta$\\%
\L(.31*.15) $U_3^+$\\%
\L(.36*.15) $U_4^+$\\%
\endSetLabels
\begin{center}
\AffixLabels{\centerline{\includegraphics[width=7cm]{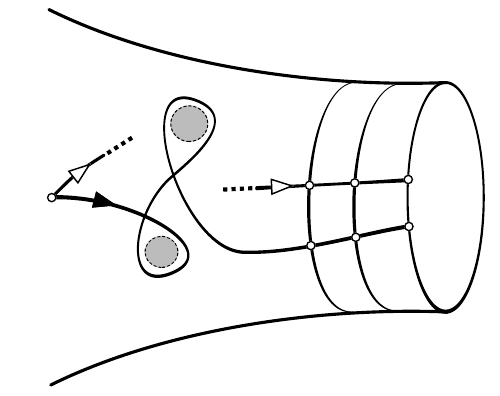}\quad\quad\quad\quad\includegraphics[width=7cm]{Figures/BCylinder1.pdf}}}
\vspace{-24pt}
\end{center}
\caption{The cylinder $U^+$}
\label{fig:cylinder1}
\end{figure}
		
We choose a simple curve $\eta_m$ in the cylinder $U^+$ which intersect each arc of $\alpha\cap U^+$ once and denote its intersections with $e_2$ and $e_2'$ by $p_m$ and $p_m'$ respectively. We denote by $c_m$ the segment in $\eta_m$ going from $p_m$ to $p_m'$ following the orientation induced by $\eta$.

Let $e_3$ (resp. $e_3'$) be the path from $p_+$ to $p_m$ (resp. from $p_+'$ to $p_m$) and $e_4$ (resp. $e_4'$) be the path from $p_m$ to $p$ (resp. from $p_m'$ to $p'$). The curve $\eta_m$ separates $U^+$ into two cylinders, and we denote by $U_3^+$ and $U_4^+$ the cylinders containing $e_3$ and $e_4$ respectively. We use $\eta_3^+$ (resp. $\eta_4^+$) to denote a simple closed curve in $U_3^+$ (resp. $U_4^+$), which intersects each arc of $\alpha$ in $U_3^+$ (resp. $U_4^+$) once.

Notice that the loops
	\[
		c_+\ast e_2'\ast c^{-1}\ast e_2^{-1},\quad c_+\ast e_3'\ast c_m^{-1}\ast e_3^{-1},\quad c_m\ast e_4'\ast c^{-1}\ast e_4^{-1}
	\]
are all homotopically trivial and hence
	\[
		e_1'\ast(e_1\ast c_+)^{-1}\sim e_1'\ast e_3'\ast(e_1\ast e_3\ast c_m)^{-1}\sim \eta^k.
	\]
	
Now we modify the segment $e_3'\ast e_4'$ in the following way. We replace $e_3'$ and $e_4'$ by 
	\[
		e''_3=e_3'\ast\eta_m^{-k}\quad\textrm{and}\quad e''_4=\eta^k_m\ast e'_4.
	\]
respectively (see Figure \ref{fig:cylinder2}).

\begin{figure}[h]
\leavevmode \SetLabels
\L(.34*.43) $e_1$\\%
\L(.34*.63) $e_1'$\\%
\L(.585*.14) $e_3''$\\%
\L(.635*.14) $e_4''$\\%
\L(.30*.52) $q$\\%
\L(.595*.47) {\small $c_m$}\\%
\L(.65*.47) {\small $c$}\\%
\L(.625*.83) $\eta_m$\\%
\endSetLabels
\begin{center}
\AffixLabels{\centerline{\includegraphics[width=7cm]{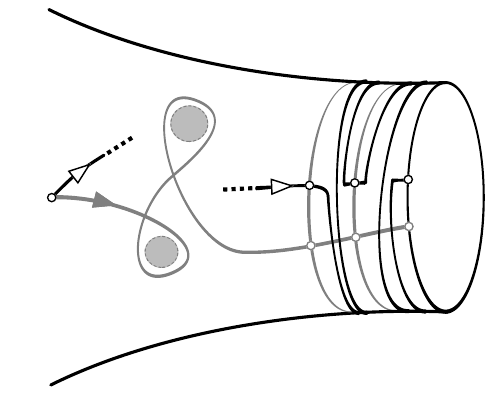}}}
\vspace{-24pt}
\end{center}
\caption{Modifying the representative of the curve}
\label{fig:cylinder2}
\end{figure}
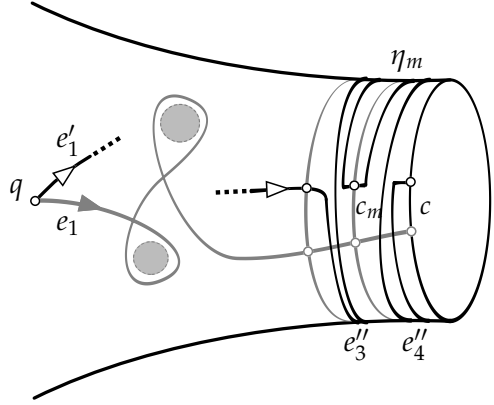
	
This gives
\[
\begin{aligned}
&(e_1\ast e_3\ast c_m)^{-1}\ast e_1'\ast e_3''\\
\sim&(e_1\ast e_3\ast c_m)^{-1}\ast e_1'\ast e_3'\ast\eta_m^{-k}\\
\sim&\eta_m^k\ast\eta_m^{-k}\\
\sim&\{\bullet\}
\end{aligned}
\]
and
\[
\begin{aligned}
&(e_1\ast e_3)^{-1}\ast e_1'\ast e_3''\ast e_4''\ast(e_4\ast c)^{-1}\\
\sim&[(e_1\ast e_3)^{-1}\ast e_1'\ast e_3''\ast c_m^{-1}]\ast [c_m\ast \eta^k_m\ast e_4'\ast(e_4\ast c)^{-1}]\\
\sim&\eta_m^k\sim\eta^k
\end{aligned}
\]

By the above discussion, we have $e_1\ast e_3\ast c_m\sim e_1'\ast e_3''$. We replace $e'$ by $e_1\ast e_3\ast c_m\ast e_4''$ and get a new representative of $\alpha$ (see Figure \ref{fig:cylinder3}).

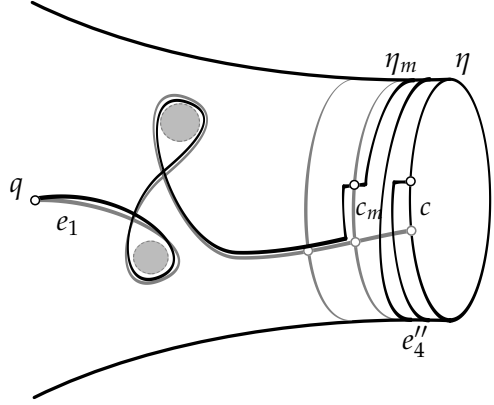
\begin{figure}[h!]
\leavevmode \SetLabels
\L(.34*.43) $e_1$\\%
\L(.635*.145) $e_4''$\\%
\L(.30*.52) $q$\\%
\L(.595*.465) {\small$c_m$}\\%
\L(.65*.465) {\small$c$}\\%
\L(.62*.82) $\eta_m$\\%
\L(.68*.82) $\eta$\\%
\endSetLabels
\begin{center}
\AffixLabels{\centerline{\includegraphics[width=7cm]{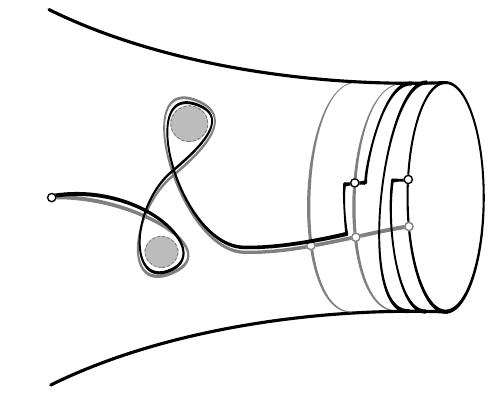}}}
\vspace{-24pt}
\end{center}
\caption{Representing $\alpha$}
\label{fig:cylinder3}
\end{figure}

If we have $e'\ast(e\ast c)\sim\eta^k$ with $k<0$, then the twisting direction of $e_4''$ will be opposite to the orientation of $c_m$. In this case, after we finish the above steps, we perform another homotopy $\overline{c_m}\sim c_m\ast\eta_m^{-1}$ where $\overline{c_m}$ is the segment in $\eta$ going from $p$ to $p'$ against the orientation of $\eta$. Hence the arc $e_4'$ is replaced by $\overline{c_m}\ast \eta_m^{k+1}\ast e_4'$ (see Figure \ref{fig:cylinder4}).
	
\begin{figure}[h!]
\leavevmode \SetLabels
\L(.325*.50) $c_m$\\%
\L(.585*.72) $\overline{c_m}$\\%
\endSetLabels
\begin{center}
\AffixLabels{\centerline{\includegraphics[width=8cm]{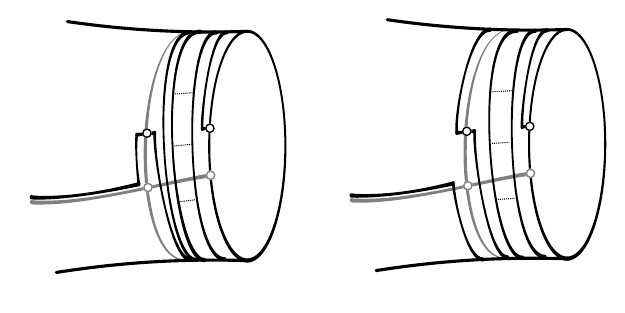}}}
\vspace{-24pt}
\end{center}
\caption{Replacing $e_4'$}
\label{fig:cylinder4}
\end{figure}
	
We denote by $e^\dag$ be the segment that we use to replace $e'$. We consider the segment $h=e_1\cup e_3$ shared by $e$ and $e^\dag$. We extend $h$ in disk neighborhoods of its endpoints $q$ and $p_m$, such that there is no new intersection point appearing, and denote by $h_e$ and $h_{e^\dag}$ the extension of $h$ along $e$ and $e^\dag$ respectively. We would like to count $i(h_e,h_{e^\dag})$.
	
Let $D$ (resp. $D_m$) denote the disk neighborhood of $q$ (resp. $p_m$) with boundary $S$ (resp. $S_m$). We denote by $I=S\setminus h$ and $I_m=S_m\setminus h$. The surface orientation induces an orientation on $S$ and $S_m$, which moreover induces an orientation on $I$ and $I_m$ which we call the positive direction. 
	
We denote by $x$ and $x_m$ (resp. $x^\dag$ and $x_m^\dag$) the intersection between $h_{e}$ (resp. $h_{e^\dag}$) with $I$ and $I_m$ respectively (see Figure \ref{fig:intersection}).
\begin{claim}\label{intersectionnumber}
If the orientation from $x$ to $x^\dag$ and the one from $x_m$ to $x_m^\dag$ have the same sign, then $i(h_e,h_{e^\dag})=i(h,h)+1$;
		
If the orientation from $x$ to $x^\dag$ and the one from $x_m$ to $x_m^\dag$ have different signs, then $i(h_e,h_{e^\dag})=i(h,h)$.
\end{claim}
\begin{remark}
As a convention, when we talk about intersection between two segments, we consider the homotopy relative to the endpoints of the segments, and count only the intersections in the interior of the segments.
\end{remark}

\begin{proof}[Proof of Claim \ref{intersectionnumber}]
We first consider the case where the $h$ is simple. Let $U_h$ be a neighborhood of $h_e\cup h_{e^\dag}$ obtained by taking the union of $D$ and $D_m$ with a well-chosen disk neighborhood $D_h$ of $h$, such that
\begin{itemize}
	\item	$U_h$ is a topological disk;
	\item	$\partial D\cap U_h$ and $\partial D_m\cap U_h$ are intervals;
	\item	the four points $x$, $x_m$, $x^\dag$ and $x_m^\dag$ on the boundary of $U_h$.
\end{itemize}
Then the intersection number is $1$ if and only if the two pairs $(x,x_m)$ and $(x^\dag,x_m^\dag)$ lie on $\partial U_h$ alternatively.

Without loss of generality, we may assume that the orientation from $x^\dag$ to $x$ is always positive. We can consider the decomposition of $U_h$ into $D$ and $D^c=U_h\setminus D$ with disjoint interior. Let $J$ be the common boundary of $D$ and $D^c$. We choose two distinct points $y$ and $y^\dag$ on $J$. Notice that the orientation on $J$ induced by the orientation of $\partial D$ and the one induced by the orientation of $\partial D^c$ are different. Up to homotopy of $D$, we assume that $y$ is in the segment connecting $x$ and $x_m$. By applying the above discussion to $D$ and $D^c$ separately, we may choose $y$ and $y^\dag$ in a way such that the orientation guarantees that $i(h_e,h_{e^\dag})=i(xy,x^\dag y^\dag)$.
		
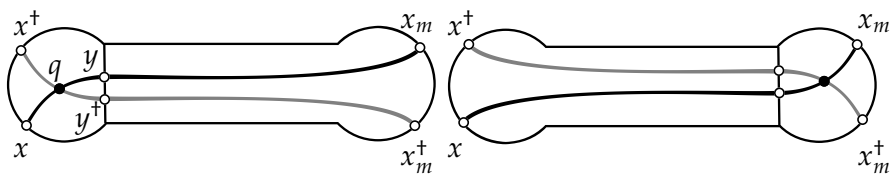
\begin{figure}[h]
\leavevmode \SetLabels
\L(.13*.85) $x^\dag$\\%
\L(.13*-.03) $x$\\%
\L(.46*.90) $x_m$\\%
\L(.46*-.09) $x_m^\dag$\\%
\L(.19*.64) $y$\\%
\L(.183*.19) $y^\dag$\\%
\L(.16*.57) $q$\\%
\L(.50*.85) $x^\dag$\\%
\L(.50*-.03) $x$\\%
\L(.85*.90) $x_m$\\%
\L(.85*-.09) $x_m^\dag$\\%
\endSetLabels
\begin{center}
\AffixLabels{\centerline{\includegraphics[width=12cm]{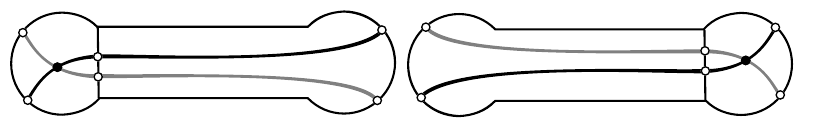}}}
\vspace{-24pt}
\end{center}
\caption{Moving an intersection point}
\label{fig:intersection}
\end{figure}

The case when $h$ has self-intersections will be taken care of by reasoning in the universal cover $\widetilde{\Sigma}$ of the surface $\Sigma$. More precisely, we consider a lift $\widetilde{h}$ of $h$ in $\widetilde{\Sigma}$, then extend it to lift $\widetilde{h_e}$ and $\widetilde{h_{e^\dag}}$ of $h_e$ and $h_{e^\dag}$ respectively. Denote by $\widetilde{D}$ and $\widetilde{D_m}$ the lifts of $D$ and $D_m$ respectively intersecting $\widetilde{h}$.
	
Considering the homotopy relative to the endpoints of $\widetilde{h_e}$ and $\widetilde{h_{e^\dag}}$, we have $i(\widetilde{h_e},\widetilde{h_{e^\dag}})$ equal to $0$ or $1$. If $i(\widetilde{h_e},\widetilde{h_{e^\dag}})=1$, we use a homotopy to get a taut representative by moving the intersection into the lift of $D$ . The part of the lift outside of the lift of $D\cup D_m$ are two disjoint segments connecting the two disks.

Notice that the restriction of the covering map on any lift of $D\cup D_m$ is injective. When we project these two segments to $\Sigma$, we get two homotopic segments (relative to $D$ and $D_m$) setwise, which are also homotopic to $h$. As such, the intersection number between them is equal to $i(h,h)$. This proves Claim \ref{intersectionnumber}.
\end{proof}

\begin{remark}\label{moveq}
The proof of the above claim shows that the intersection point in $D$ can be moved in $D_m$ while retaining minimal intersection. It also tells us that, in order to have the intersection $q$, the orientation among the intersections between $h_e\cup h_{e^\dag}$ with $I$ and $I_m$ should be equal as in the first case of the statement. 
\end{remark}
	
The following is an obvious, but useful, observation which basically says that being globally taut implies being locally taut. 
\begin{claim}\label{tautpath}
Let $\gamma$ be a taut curve. Any pair of segments intersect each other minimally up to homotopy relative to their endpoints.
\end{claim}
\begin{proof}[Proof of Claim \ref{tautpath}]
If not, there are necessarily monogons or bigons that appear, which would also be a part of $\gamma$. This contradicts the fact that $\gamma$ is taut and this proves the claim.
\end{proof}
Using the above two claims, we can conclude that 
\begin{itemize}
\item if $e'\ast(e\ast c)^{-1}\sim \eta^k$ with $k>0$, then the intersection at $q$ is positive counted from the direction of $e$ to that of $e'$; 
\item if $e'\ast(e\ast c)^{-1}\sim \eta^k$ with $k<0$, then the intersection at $q$ is negative counted from the direction of $e$ to that of $e'$;
\item if $e'\ast(e\ast c)^{-1}\sim \{\bullet\}$, then the intersection at $q$ can be either negative or positive.
\end{itemize}

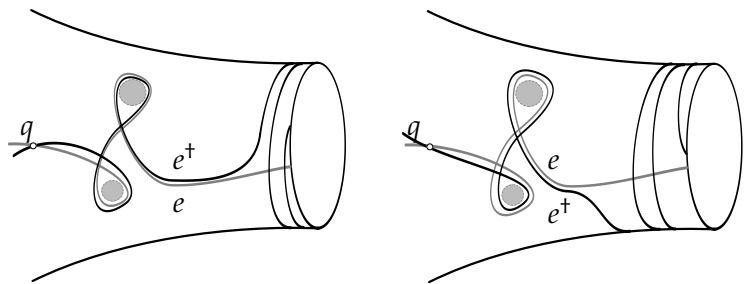
\begin{figure}[h]
\leavevmode \SetLabels
\L(.33*.42) $e^\dag$\\%
\L(.33*.27) $e$\\%
\L(.20*.54) $q$\\%
\L(.65*.42) $e$\\%
\L(.65*.23) $e^\dag$\\%
\L(.53*.54) $q$\\%
\endSetLabels
\begin{center}
\AffixLabels{\centerline{\includegraphics[width=10cm]{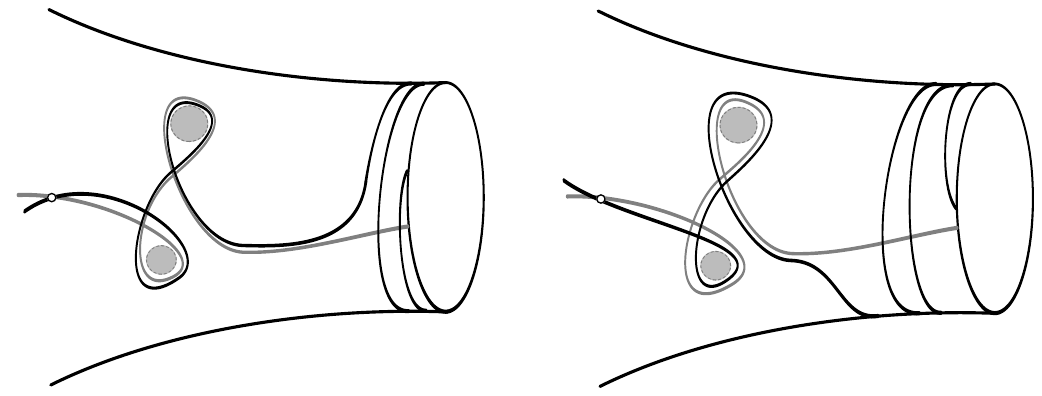}}}
\vspace{-24pt}
\end{center}
\caption{Cylinder 5}
\label{fig:cylinder5}
\end{figure}

To see this, we consider the homotopy that we take to change $e'$ to $e^\dag$. By the discussion in Remark \ref{moveq}, we can move $q$ to a neighborhood of $p_m$. When $k$ is not $0$, its sign is the same as the sign of any point in $e\cap e^\dag$ counting from $e$ to $e^\dag$. We can compare this sign with the sign of $q$. If they are different, we will obtain a bigon where one vertex is $q$. Hence, $q$ is not essential, which contradicts the fact that $\alpha$ is taut.

\medskip

Now we consider the arc $a_1$ and denote by $p$ one of its endpoints on $\eta$.We give $a_1$ the orientation pointing away from $p$. Without loss of generality, we assume that $a_1$ goes into $U^+$ once it passes $p$.

Following the orientation of $a_1$, we assume that there are intersection points $\eta$-admissible from the left and $q$ is the last intersection point on $a_1$. We label all intersection points on $a_1$ which are $\eta$-admissible from the left by $q_0=q, q_1,\dots,q_l$ following the direction from $q$ to $p$, and denote by $e$ the segment in $a_1$ between $p$ and $q$, and $e_i$ the segment in $e$ going from $q_i$ to $p$.
	
For any $q_i$, by the definition of being $\eta$-admissible, there is another segment $e_i'$ connecting $q_i$ to $\eta$ intersecting $\eta$ at $p_i$, such that the sum of the path $e_i'^{-1}\ast e_i$ and a segment of $\eta$ connecting $p$ to $p_i$ form a curve homotopic to a point of $\Sigma$ or a power of $\eta$.

By repeating the above operation, illustrated in Figure \ref{fig:cylinder2} and Figure \ref{fig:cylinder3}, we replace $e_i'$ by a path $e^\dag_i$ from $q_i$ going first along $e_i$, then twisting in $U_4^+$. By the discussion in Remark \ref{moveq}, we can moreover move $q_i$ to a disk neighborhood of $p_m$ in $U$. 

Each operation may move more than one $q_i$ into $U$. We start by doing the operation from $q_0$, then we check the remaining $q_i$s and look for the one in $\Sigma\setminus U$ with the smallest index $i$. Then we perform the operation on it. We repeat this process until it ends; and it will stop, since each $q_i$, once moved into $U$, will not move out of $U$.

We then repeat this process for all arcs in $\alpha\setminus \eta$, and in the end, we obtain a curve $\alpha_1$ homotopic to $\alpha$ which has no intersection points $\eta$-admissible from the left. 

Now consider the self-intersection points of $\alpha$ in $\Sigma\setminus U^+$ which are $\eta$-admissible from the right and the homotopy relative to $U^+$. By a similar operation as in the construction of $\alpha_1$, we obtain a curve $\alpha_2$ homotopic to $\alpha$ with no self-intersection points which are $\eta$-admissible from either the left or the right.

The last step is to perform the homotopy of $\Sigma$ relative to $\eta^+$ and $\eta^-$, or equivalently, perform a homotopy relative to the boundary in $U$ and $\Sigma\setminus U$ separately. We denote the resulting curve by $\alpha_3$. We claim that $\alpha_3$ is taut.

To see why, we consider a pair of self-intersection points of $\alpha_3$, and denote by $c_1$ and $c_2$ two segments of $\alpha_3$ connecting them. If $c_1\cup c_2$ is contained entirely in $U$ or in $\Sigma\setminus U$, it will not be a bigon due to the last step in the above construction. If one intersection point is in $U$ and the other one is in $\Sigma\setminus U$, since one is $\eta$-admissible while the other one is not, the union $c_1\cup c_2$ is not a bigon. If two intersection points are in $\Sigma\setminus U$ while $c_1$ and $c_2$ pass $U$, then $c_1\cup c_2$ is not a bigon, since otherwise the two intersection points are both $\eta$-admissible. Finally, we consider the case where two intersection points are both $\eta$-admissible, and $c_1$ and $c_2$ leave $U$ when connecting them. In this case, the union $c_1\cup c_2$ is also not a bigon. Otherwise, we may reverse the above construction and show that $\alpha$ is not taut which is a contradiction. 

This concludes the proof of Lemma \ref{cylinderlemma}. 
\end{proof}

We call such a neighborhood an \textit{$\alpha$-cylindrical neighborhood} of $\eta$. Similarly, let $M$ be a multicurve $i(M,M)\neq0$. We can define an \textit{$M$-cylindrical neighborhood} of $\eta$ in the same way. Observe that, from a topological point of view, for each pair $M$ and $\eta$, this neighborhood is unique.

\subsection{The proof of Theorem \ref{thm:fromkto0}}
We can now prove Theorem \ref{thm:fromkto0} which states that $k$-equivalence implies $0$-equivalence. 

Let $\alpha$ and $\beta$ be two $k$-equivalent curves. Given a curve $\gamma\in\C_k$, since the action of the mapping class group preserves $\C_k$, the following identity holds.
\begin{lemma}\label{MCG}
	For any $\psi\in\MCG(\Sigma)$, we have
	\[
	i(\alpha,\psi(\gamma))-i(\alpha,\gamma)=i(\beta,\psi(\gamma))-i(\beta,\gamma).
	\]
\end{lemma}

Now we go back to the proof of Theorem \ref{thm:fromkto0}. Consider $\eta$ a simple closed curve. Let $\gamma$ be a $k$-curve intersecting $\eta$ non-trivially. Consider the multicurve
	\[
		M=\{\alpha,\beta,\gamma\}.
	\]
Let $U$ be the $M$-cylindrical neighborhood of $\eta$. Let $a$, $b$ and $c$ be the restriction of $\alpha$, $\beta$ and $\gamma$ to $U$ respectively. Note that $a$, $b$ and $c$ are multiarcs, and may intersect each other. 
	
We denote by $D_\eta$ be the left Dehn twist along $\eta$ for which the cylindric neighborhood of $\eta$ is chosen to be $U$. Up to applying $D_\eta$ on $\gamma$, we assume that all arcs in $c$ intersect arcs in both $a$ and $b$ negatively. 
	
For any $m\in\mathbb{N}$, denote by
	\[
		M_m=\alpha\cup\beta\cup D_\eta^{m}(\gamma).
	\]
Since $U$ is $M$-cylindrical neighborhood, and by a straightforward use of the bigon criterion, the multicurve $M_m$ is still taut with $U$ its cylindrical neighborhood. 
	
As in Lemma \ref{lem:arcs}, for $m=1$, we have:
\[
	\begin{aligned}
		i(\alpha,D_\eta(\gamma))-i(\alpha,\gamma)&=i(a,\eta)i(c,\eta),\\
		i(\beta,D_\eta(\gamma))-i(\beta,\gamma)&=i(b,\eta)i(c,\eta).
	\end{aligned}
\]
By Lemma \ref{MCG}, we have
	\[
		i(a,\eta)i(c,\eta)=i(b,\eta)i(c,\eta),
	\]
since
	\[
		i(c,\eta)=i(\gamma,\eta)\neq0,
	\]
we have the identity
	\[
		i(a,\eta)=i(b,\eta),
	\]
which moreover induces
	\[
		i(\alpha,\eta)=i(\beta,\eta).
	\]
As the above holds for any simple closed curve $\eta$, this completes the proof of Theorem \ref{thm:fromkto0}. 
\section{Pairs of non-homotopic curves}\label{sec:pants}
Although for any integer $k>0$, two $k$-equivalent curves cannot be distinguished by their intersections with simple curves, it is still possible to tell them apart by considering their intersection with other curves. 

This can be roughly understood with the help of hyperbolic geometry. The Birman-Series theorem \cite{Birman-Series85} tells use there is lot of space in the complement of the union of closed geodesics with bounded self-intersection number in a hyperbolic surface. On the other hand, a random geodesic goes everywhere on a hyperbolic surface. Hence, to be able to see everywhere on the surface, we have to consider all closed geodesics, and in particular those with all possible self-intersection numbers.

In this section, we prove that distinct curves on $\Sigma$ are not $k$-equivalent for infinitely many integers $k$, as stated in Theorem \ref{thm:otal} in the introduction. The main technical tool used in this section what we call the "pants lemma" which can be considered as an enhanced version of Lemma \ref{cylinderlemma}. We introduce and prove this result in the following subsection.

\subsection{The pants lemma}

We begin with an informal description of the pants lemma. Given a closed curve and a pants decomposition, you can look at the arcs formed by the curves on the individual pairs of pants. Sometimes these arcs are simple, sometimes they have self-intersection, and sometimes they pairwise intersect. The lemma states that, given a curve, there exists a pants decomposition such that the arcs we see on the pants are all simple and pairwise disjoint. That means that any self intersection points of the curve come from pasting conditions of these arcs in cylinder neighborhoods of the pants curves. Moreover, the arcs in each cylinder will intersect each other at most once.

\begin{figure}[h]
\leavevmode \SetLabels
\endSetLabels
\begin{center}
\AffixLabels{\centerline{\includegraphics[width=12cm]{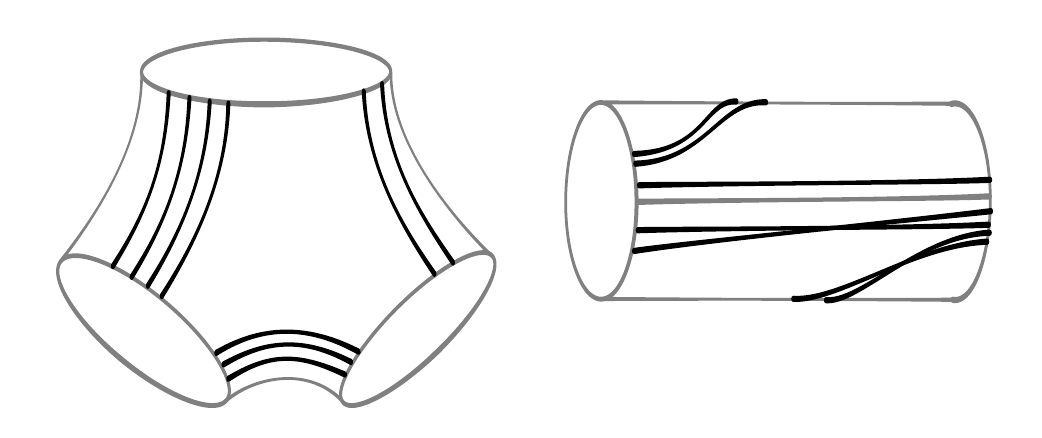}}}
\vspace{-24pt}
\end{center}
\caption{An illustration of Lemma \ref{lem:pants}}
\label{fig:pantslemma}
\end{figure}
A more formal statement is as follows:
\begin{lemma}[The pants lemma]\label{lem:pants}
For any closed curve $\alpha$, there exists a pants decomposition $P=\{\eta_1,\hdots,\eta_{3g-3}\}$ of $\Sigma$, and taut representatives of $\alpha$ and $P$ such that for $i=1,\hdots,3g-3$ there are cylindric neighborhoods $U_i$ of $\eta_i$ satisfying:
\begin{enumerate}
\item For $i\neq j$, $U_i$ and $U_j$ are disjoint.
\item For each $i$, there is a non-trivial arc of $U_i$ disjoint from $\alpha\cap U_i$.
\item All connected components of $\alpha\cap(\Sigma\setminus \cup U_i)$ are pairwise disjoint, simple and connecting distinct boundary components of a pair of pants.
\end{enumerate}
\end{lemma}

\begin{proof}
Consider a pair of pants decomposition
	\[
		P=\{\eta_1,\dots,\eta_{3g-3}\}
	\]
of $\Sigma$.

\begin{figure}[h]
\leavevmode \SetLabels
\endSetLabels
\begin{center}
\AffixLabels{\centerline{\includegraphics[width=12cm]{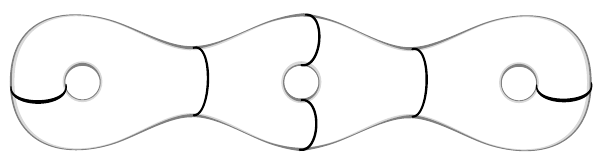}}}
\vspace{-24pt}
\end{center}
\caption{A pants decomposition of a genus $3$ surface}
\label{fig:genus3-1}
\end{figure}
Let $\Gamma$ be a dual graph of $P$ embedded in $\Sigma$. Let $\Sigma^+$ denote its neighborhood which can retract to it. We denote by $\Sigma^-$ its complement in $\Sigma$. We choose $V_i$ a cylindric neighborhood of $\eta_i$, such that for $i\neq j$, $V_i\cap V_j=\emptyset$. The complement of
	\[
		\bigcup_{i=1}^{3g-3}V_i
	\]
in $\Sigma$ is a collection of pair of pants denoted by 
	\[
		\{Y_1\dots,Y_{2g-2}\}.
	\]
In each pair of pants $Y_i$, we denote by $H^+_i$ and $H^-_i$ its intersection with $\Sigma^+$ and $\Sigma^-$ respectively.

\begin{figure}[h]
\leavevmode \SetLabels
\L(.40*.75) $\Sigma^+$\\%
\L(.77*.90) $\Sigma^-$\\%
\endSetLabels
\begin{center}
\AffixLabels{\centerline{\includegraphics[width=12cm]{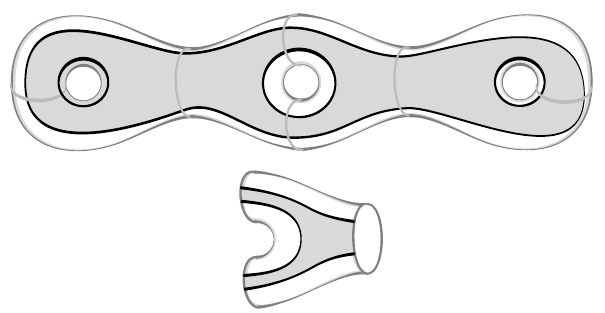}}}
\vspace{-24pt}
\end{center}
\caption{Cutting pants into hexagons}
\label{fig:genus3-cutinhalf}
\end{figure}

\begin{figure}[h]
\leavevmode \SetLabels
\endSetLabels
\begin{center}
\AffixLabels{\centerline{\includegraphics[width=12cm]{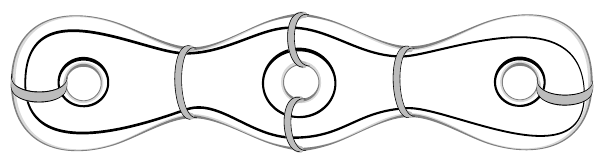}}}
\vspace{-24pt}
\end{center}
\caption{Cylinders around pants curves}
\label{fig:genus3-cylindrical}
\end{figure}

The two subsurface $\Sigma^+$ and $\Sigma^-$ share the same boundary whose connected components form a multicurve. We denote it by
	\[
		\{\zeta_1,\dots,\zeta_m\}.
	\]
Since they are disjoint, their associated Dehn twists commute with each other. We denote the composition of their associated left Dehn twists by 
\[
	\psi=D_{\zeta_1}\circ\cdots\circ D_{\zeta_m}.
\]

We assume that $\alpha$ intersects $\partial V_i$ and the $\zeta_i$s minimally and
	\[
		(\alpha\cap\partial V_i)\cap(\alpha\cap\zeta_j)=\emptyset,
	\]
for all possible $i$ and $j$. 
	
By cutting $\Sigma$ along curves in the $\partial V_i$s and the $\zeta_j$s, we cut $\alpha$ into segments. The boundary of each $Y_i$ is cut into $6$ connected components. By considering where the endpoints of the segment $\alpha\cap Y_i$ lands in these 6 boundary connected components, we list all possible topological types of the segment $\alpha\cap Y_i$, as well as the corresponding types of their $\psi$-images.

\begin{figure}[h]
\leavevmode \SetLabels
\endSetLabels
\begin{center}
\AffixLabels{\centerline{\includegraphics[width=18cm]{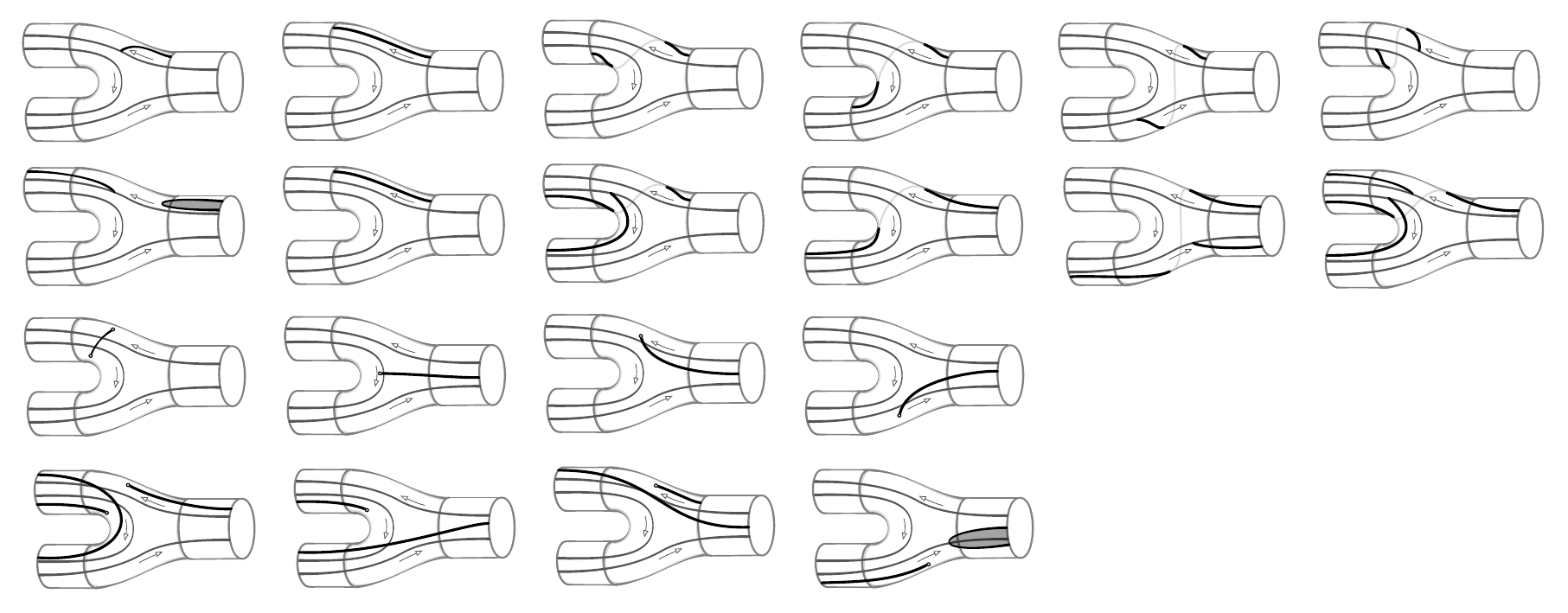}}}
\vspace{-24pt}
\end{center}
\caption{The topological cases for arcs on a pair of pants}
\label{fig:TopoCasesPants}
\end{figure}

Similarly, for each $V_i$, we list all possible topological types of the segment $\alpha\cap V_i$, as well as the corresponding types of their $\psi$-images.

\begin{figure}[h]
\leavevmode \SetLabels
\endSetLabels
\begin{center}
\AffixLabels{\centerline{\includegraphics[width=18cm]{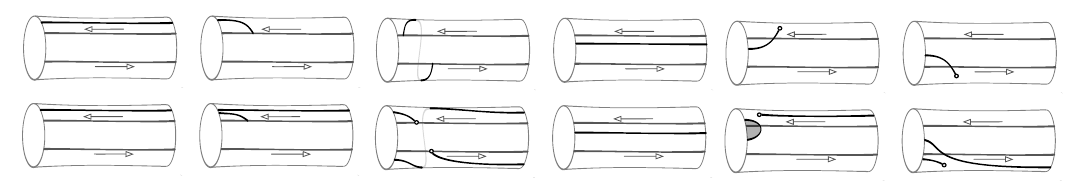}}}
\vspace{-24pt}
\end{center}
\caption{The topological cases for arcs on a cylinder}
\label{fig:TopoCasesCylinder}
\end{figure}

There are two issues that we have to consider. The first is that the $\psi$-image of $\alpha\cap Y_i$ or $\alpha\cap V_i$ may form a bigon with the $\eta_i$s. Secondly the segment $\alpha\cap Y_i$ or its image may not entirely contained in $H_i^+$.

To solve the first issue, instead of directly apply $\psi$, we start by right twisting $\alpha$ along pants curves $\eta_i$s for some times, so that the intersections $\alpha$ with the $\zeta_j$s are all positive. Using Lemma \ref{lem:arcs}, the $\psi$-image of $\alpha\cap Y_i$ or $\alpha\cap V_i$ cannot form a bigon with the $\eta_i$s any more. From now on, we assume that $\psi$ is modified by Dehn twists along the $\eta_i$s so that the first issue is solved.

To solve the second issue, notice that for each $i$, the $\psi$-image $\alpha\cap Y_i$ passes $Y_i$ from one boundary to a different one. Hence we may move it into $\Sigma^+$ and the cost is that there may be additional twists in the neighbor cylinder part.

\medskip

Now we would like to show that the pants decomposition $\psi^{-1}(P)$ is the one that we are looking for. The proof is essentially the same as in Lemma \ref{cylinderlemma}. The only ambiguity is that a intersection point of $\alpha$ may be admissible to several pants curves. In that case, we choose one once for all to which cylinder we would like to move that admissible intersection point. 
\end{proof}
\begin{remark}
As for Lemma \ref{cylinderlemma}, the pants lemma also applies to a multicurve. Any pants decomposition that satisfies the pants lemma for a multicurve $M$ is called a \textit{$M$-cylindrical pants decomposition}. \end{remark}

\subsection{Intersections identify curves}\label{IIC}
This part is dedicated to the proof of Theorem \ref{thm:otal} which recall below:

\begin{theorem-otal}
	Any pair of distinct curves are not $k$-equivalent for infinitely many $k$.
\end{theorem-otal}
\begin{proof}
Let $\alpha$ and $\beta$ be distinct curves. Let $M=\{\alpha,\beta\}$. By the pants lemma, we can find a $M$-cylindrical pants decomposition $P$. We use the same notation as in the previous section. Let $\eta_i$s denote the pant curves and $U_i$s denote the cylindric neighborhood of the $\eta_i$s. Let $\Sigma^+$ denote a subsurface of $\Sigma$ which can contract to a dual graph of $P$ and $\Sigma^-$ be its complementary. Let the curves of the boundary of $\Sigma^+$ (as well as $\Sigma^-$) be denoted by $\zeta_i$. 

By the property of $P$, the intersection between $M$ and $\Sigma^-$, if it exists, are arcs homotopic to the intersection between $\eta_i$s and $\Sigma^-$. Given any curve $\gamma\subset\Sigma^-$, we have
	\[
		i(\gamma,\alpha\cup\beta)=i(\gamma,(\alpha\cup\beta)\cap\Sigma^-).
	\]
This is due to the fact that all arcs in $(\alpha\cup\beta)\cap\Sigma^-$ intersect the $\zeta_i$s negatively. By the bigon criterion, we have equality. Thus we can consider the problem in two cases, depending on whether $\alpha\cap \Sigma^-$ and $\beta\cap\Sigma^-$ are the same as multiarcs in $\Sigma^-$.

\begin{figure}[h]
\leavevmode \SetLabels
\endSetLabels
\begin{center}
\AffixLabels{\centerline{\includegraphics[width=12cm]{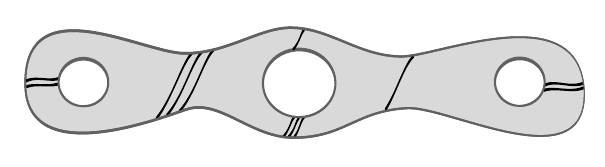}}}
\vspace{-24pt}
\end{center}
\caption{A curve seen on $\Sigma^{-}$}
\label{fig:genus3-onehalf}
\end{figure}

{\bf Case 1.} $\alpha\cap\Sigma^-$ and $\beta\cap\Sigma^-$ are different as multiarcs in $\Sigma^-$.

The key tool used in the proof for this case is the rigidity result for measured hexagon decomposition of bordered surface, whose construction is inspired by the construction of measure foliations as a generalization of simple curves. Using this result, we find a simple curve $\gamma$ in $\Sigma^-$ which intersects $\alpha$ and $\beta$ differently. Details are provided in Appendix \ref{app:hexagon}. 

Before going further, we point out that we are going to construct curves and it will be convenient to consider all curves as oriented objects. Since $\gamma$ is a simple curve in $\Sigma^-$, there is a boundary component $\zeta_i$ which is not homotopic to $\gamma$. Connecting $\gamma$ with $\eta_i$ by a simple curve disjoint from them, we may construct a new curve $\gamma^k\zeta_i$ by first going along $\gamma$ for $k>0$ times, then going along the chosen arc to $\zeta_i$ and following $\zeta_i$ once. Assume $\gamma\zeta_i$ is a $1$-curve. Then for any $k\in\mathbb{N}$, the curve $\gamma^k\zeta_i$ is a $k$-curve.

The proof in Appendix \ref{app:hexagon} shows that for any integer $k>0$, we have
	\[
		i(\alpha,\gamma^k\zeta_i)=\max\{ki(\alpha,\gamma)+i(\alpha,\zeta_i),(k-1)i(\alpha,\gamma)+i(\alpha,\gamma^{-1}\zeta_i)\}.
	\]
where $\gamma^{-1}\zeta_i$ is obtained by first going along the reverse of $\gamma$, then following the chosen arc to $\zeta_i$ and going along $\zeta_i$ once. Since $\gamma\zeta_i$ is $1$-curve, the curve $\gamma^{-1}\zeta_i$ is simple. Hence there exists a constant $m_0\in\mathbb{N}$, such that
	\[
		\begin{aligned}
			|i(\alpha,\gamma^k\zeta_i)-ki(\alpha,\gamma)|&\le m_0,\\
			|i(\beta,\gamma^k\zeta_i)-ki(\beta,\gamma)|&\le m_0.
		\end{aligned}
	\]

By hypothesis we have
	\[
		i(\alpha,\gamma)\neq i(\beta,\gamma),
	\]
and hence there exists an integer $k_0>0$, such that for any integer $k>k_0$, we have
	\[
		i(\alpha,\gamma^k\zeta_i)\neq i(\beta,\gamma^k\zeta_i).
	\]
Therefore, for any integer $k>k_0$, the curves $\alpha$ and $\beta$ are not $k$-equivalent. This concludes the proof in this first case.

{\bf Case 2.} $\alpha\cap\Sigma^-$ and $\beta\cap\Sigma^-$ are the same as multiarcs in $\Sigma^-$

Now suppose that all curves in $\Sigma^-$ intersect $\alpha\cap\Sigma^-$ the same number of times as $\beta\cap\Sigma^-$. We look at $\Sigma^+$; since $\alpha$ and $\beta$ are not homotopic to each other, there exists a cylinder $U_i$ such that at least one of the following holds:
	\begin{enumerate}
		\item $\alpha\cap U_i$ and $\beta\cap U_i$ have a different number of arcs,
		\item $\alpha\cap U_i$ and $\beta\cap U_i$ have the same number of arcs, but with different configurations of self-intersections.
	\end{enumerate}
We analyze each case.
\paragraph{Case 2.1. $\alpha\cap U_i$ and $\beta\cap U_i$ have a different number of arcs.}
This case can be treated in a similar fashion to the previous one. The number of arcs in $\alpha\cap U_i$ and $\beta\cap U_i$ are the intersection numbers $i(\alpha,\eta_i)$ and $i(\beta,\eta_i)$. Hence in this case, we have
	\[
		i(\alpha,\eta_i)\neq i(\beta,\eta_i).
	\]
	
We consider an arc $a$ in $U_i$. Up to applying the Dehn twist $D_{\eta_i}$ to the cylinder enough times, we may assume that $a$ intersects all arcs in $\alpha\cap U_i$ and $\beta\cap U_i$ negatively. We then have:
	\[
		\begin{aligned}
			i(D_{\eta_i}^{k+1}(a),\alpha\cap U_i)-i(D_{\eta_i}^{k}(a),\alpha\cap U_i)&=i(\eta_i,\alpha)\\
			i(D_{\eta_i}^{k+1}(a),\beta\cap U_i)-i(D_{\eta_i}^{k}(a),\beta\cap U_i)&=i(\eta_i,\beta)
		\end{aligned}
	\]
Hence we have
	\[
		i(D_{\eta_i}^{k}(a),\alpha\cap U_i)\neq i(D_{\eta_i}^{k}(a),\beta\cap U_i).
	\]

Now we connect the two ends of $a$ by a path $b$ passing this cylinder, such that $a\cup b$ is a curve with intersection number $2$ with $\eta_i$ (see Figure \ref{fig:twisting}). We denote this curve by $\gamma$. By replacing $a$ by $D^k_{\eta_i}(a)$, we obtain a new curve denoted by $\gamma_k$, which is a $k$ curve.

Assume that $b$ intersects $\alpha$ minimally. Consider the intersection number $i(\alpha,\gamma_k)$, and we have
	\[
		i(\alpha,\gamma_k)\le i(\alpha,b)+i(\alpha, D^k_{\eta_i}(a)).
	\]
Notice that the inequality holds only if there is an $\eta_i$-admissible intersection in $\alpha\cap b$. Hence there exits a positive constant $m_0$, such that
	\[
		|i(\alpha,\gamma_k)-i(\alpha, D^k_{\eta_i}(a))|\le m_0.
	\]
Similarly, we have a positive constant $m_1$ such that
	\[
		|i(\beta,\gamma_k)-i(\beta, D^k_{\eta_i}(a))|\le m_1.
	\]
Therefore, we have a positive integer $k_0$ such that, for any $k\ge k_0$, we have
	\[
		i(\alpha,\gamma_k)\neq i(\beta,\gamma_k).
	\]
\begin{figure}[h!]
\leavevmode \SetLabels
\endSetLabels
\begin{center}
\AffixLabels{\centerline{\includegraphics[width=8cm]{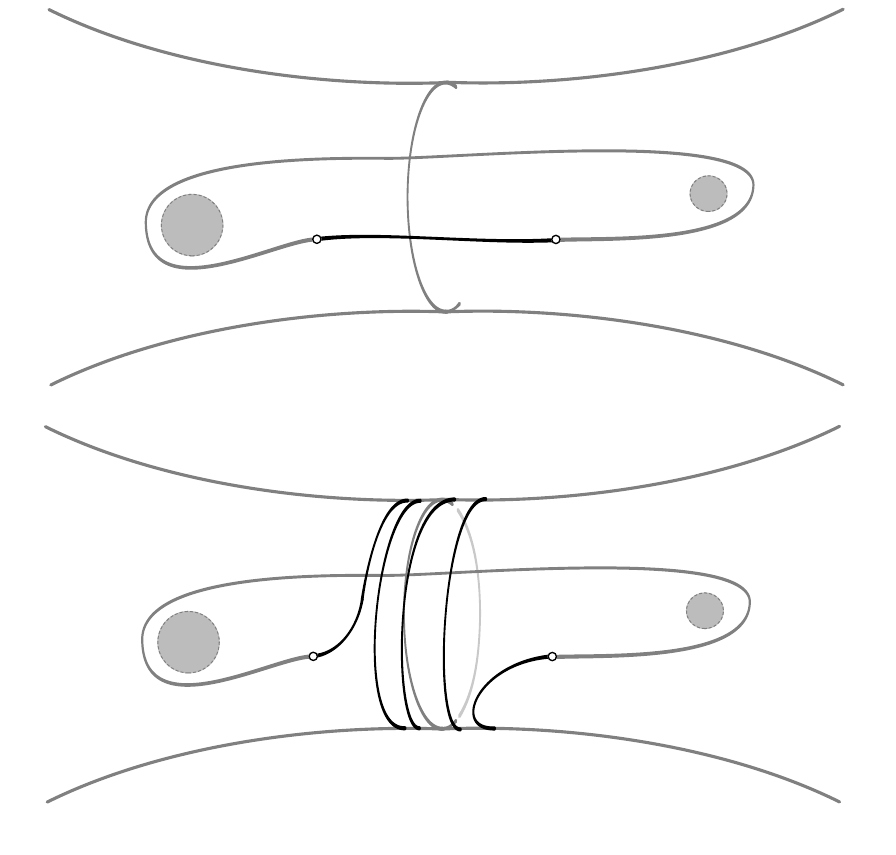}}}
\vspace{-24pt}
\end{center}
\caption{Twisting an arc to add intersection}
\label{fig:twisting}
\end{figure}

\paragraph{Case 2.2. $\alpha\cap U_i$ and $\beta\cap U_i$ have the same number of arcs, but with different configurations of arcs.}
The construction of curves in this case still starts from an arc in $U_i$. In this case, $\alpha\cap U_i$ and $\beta\cap U_i$ have the same number of arcs or, equivalently,
	\[
		i(\alpha,\eta_i)=i(\beta,\eta_i).
	\]
Now we would like to be more precise about the difference between the configurations of $\alpha\cap U_i$ and $\beta\cap U_i$. For each $i$, we choose an orientation on $\eta_i$ and can consider the left side or the right side of $\eta_i$. We denote by $U_i^+$ (resp. $U_i^-$) the part of $U_i$ on the left (resp. right) of $\eta_i$ and denote its boundary different from $\eta_i$ by $\eta_i^+$ (resp. $\eta_i^-$).

Notice that the intersection point could move from a cylinder to another cylinder. Let $a_j$ and $a_l$ be two arcs of $\alpha\cap U_i$ intersecting at the point $p\in U_i$. We extend $a_j$ and $a_l$ from the same side of $U_i$. We say they are parallel until step $m\in\mathbb{N}$ if they have extensions passing successively through a same sequence of $m$ subsurfaces in the pants-cylinder decomposition, and when connecting the endpoints of their extension with a segment in the last subsurface where they arrived, this segment with the the extensions of the two arcs form a homotopically trivial curve. Then the intersection point $p$ can be moved via homotopy to any cylinder where this pair of parallel extensions of $a_j$ and $a_l$ passed through.

\begin{figure}[h]
	\leavevmode \SetLabels
	\endSetLabels
	\begin{center}
		\AffixLabels{\centerline{\includegraphics[width=8cm]{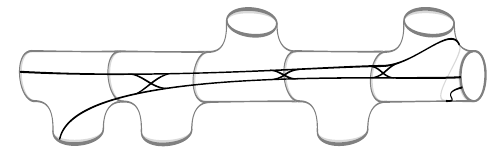}}}
		\vspace{-24pt}
	\end{center}
	\caption{Different crossing points give the same curve}
	\label{fig:crosses}
\end{figure}

Using our previous terminology, when this intersection point $p$ is admissible with respect to multiple $\eta_i$s in $P$, we can move $p$ to any cylinder associated to one of these curves in $P$.

When we try to compare $\alpha\cap U_i$ and $\beta\cap U_i$, we will first perform a homotopy to move all self-intersections of $\alpha$ and those of $\beta$ which are admissible to $\eta_i$ to $U_i$. Moreover, if an intersection point is $\eta_i$-admissible from both sides of $\eta_i$, then we move it to $U_i^+$. Then, since $\alpha$ and $\beta$ are not homotopic to each other, there must be some index $i$ such that, when comparing as above, $\alpha\cap U_i$ and $\beta\cap U_i$ are not homotopic to each other (with endpoints of arcs moving along $\partial U_i\cap \Sigma^+$ in the homotopy). We will consider this $U_i$ in what comes next.

The intersection $\Sigma^+\cap U_i$ is a topological quadrilateral $Q$. Let $v,v'\in \eta_i^+$ and $w,w'\in\eta_i^-$ be the vertices of $Q$ such that $v,v',w',w$ follows the clockwise cyclic order on $\partial Q$. We denote the intersection points in $\alpha\cap \eta_i^+$ by $v_1,...,v_n$ following the orientation from $v$ to $v'$, and denote the intersection points in $\alpha\cap \eta_i^-$ by $w_1,...,w_n$ following the orientation from $w$ to $w'$. To compare $\alpha\cap U_i$ and $\beta\cap U_i$, by taking a homotopy in $U_i$ (allowing endpoints of arcs move along $\partial U_i\cap \Sigma^+$), we may assume that arcs in $\beta\cap U_i$ also connect the $v_i$s to the $w_i$s. By hypothesis, $\alpha\cap U_i$ is not homotopic to $\beta\cap U_i$ relative to $\partial U_i$.

Now we would like to construct a curve which intersects $\alpha$ and $\beta$ differently. We start by considering the arc $a$ intersecting $vv_1$ and $w_1w_2$ which intersects $\alpha\cap U_i$ and $\beta\cap U_i$ negatively and with the same intersection number.

\begin{figure}[h]
	\leavevmode \SetLabels
	\L(.41*.9) $v'$\\%
	\L(.395*.25) $v$\\%
	\L(.59*.9) $w'$\\%
	\L(.605*.25) $w$\\%
	\endSetLabels
	\begin{center}
		\AffixLabels{\centerline{\includegraphics[width=4cm]{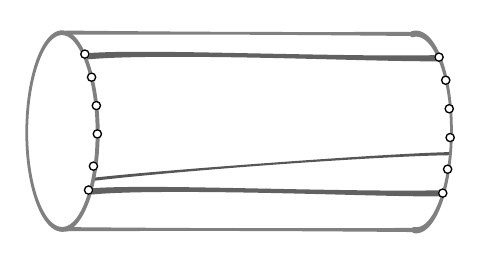}}}
		\vspace{-24pt}
	\end{center}
	\caption{The cylinder $U_i$}
	\label{fig:cylindersimple}
\end{figure}
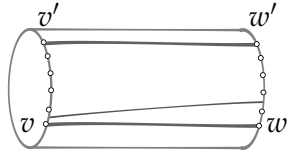

Here, for each of the two curves, the extension from $w_1$ and $w_2$ will be separated by topology eventually. At the same time, the extensions of $\alpha$ and $\beta$ will also be separated, otherwise they would be the same curve. 

\begin{figure}[h!]
	\leavevmode \SetLabels
	\endSetLabels
	\begin{center}
		\AffixLabels{\centerline{\includegraphics[width=7cm]{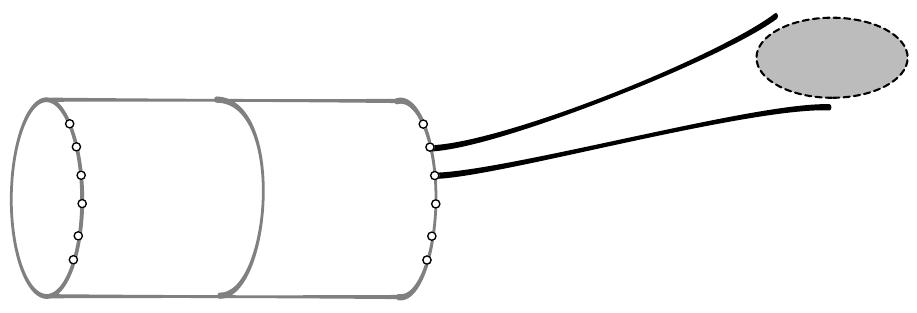}}}
		\vspace{-24pt}
	\end{center}
	\caption{Arcs are separated by topology eventually}
	\label{fig:cylindertop}
\end{figure}

Figure \ref{fig:cylindermultiple} only shows 4 cases, and there are 4 other cases given by exchanging $\alpha$ and $\beta$.

\begin{figure}[h!]
\leavevmode \SetLabels
\endSetLabels
\begin{center}
\AffixLabels{\centerline{\includegraphics[width=9cm]{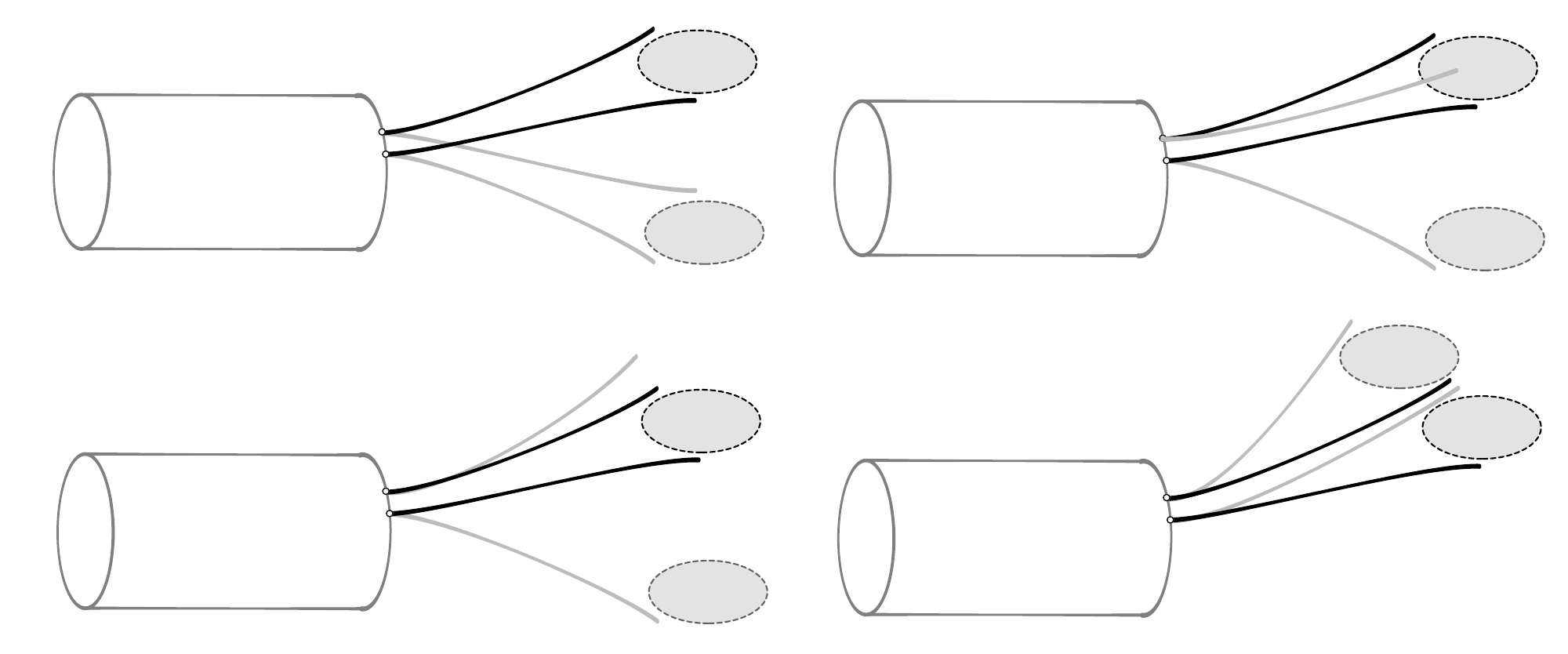}}}
\vspace{-24pt}
\end{center}
\caption{The first four cases}
\label{fig:cylindermultiple}
\end{figure}

Denote by $\alpha_1$ and $\alpha_2$ (resp. $\beta_1$ and $\beta_2$) the extension of $\alpha$ (resp. $\beta$) from $w_1$ and $w_2$ respectively. Without loss of generality, we may assume that $\alpha_1$ is the right most extension, otherwise we may exchange the role of $\alpha$ and $\beta$. Then, if we extend $a$ following $\alpha_1$, this extension, denoted by $a_-$, will intersect $\beta_1$. On the other hand, the arc $a$ already intersects the arc connecting to $w_1$, hence we have a bigon. Therefore, when we move the path $a_-$ to a minimal position with $\beta$, it should intersect $\eta_i^-$ between $w$ and $w_1$ (see Figure \ref{fig:cylindertwo}).

\begin{figure}[h]
\leavevmode \SetLabels
\endSetLabels
\begin{center}
\AffixLabels{\centerline{\includegraphics[width=7cm]{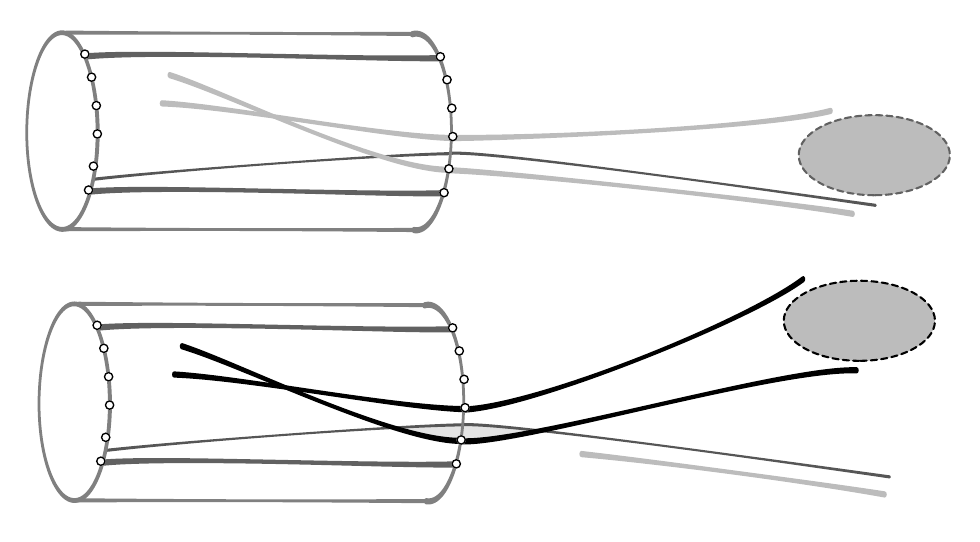}}}
\vspace{-24pt}
\end{center}
\caption{Intersecting between $w$ and $w_1$}
\label{fig:cylindertwo}
\end{figure}

Notice that the number of arcs in $\alpha\cap U_i$ intersecting $\Sigma^-$ is the same as that for $\beta$. We denote this number by $s$. Hence $a_-$ has $s+1$ intersections with $\alpha$ and has $s$ intersections with $\beta$ in $U_i$ (see Figure \ref{fig:cylinderintersections}).

\begin{figure}[h]
\leavevmode \SetLabels
\endSetLabels
\begin{center}
\AffixLabels{\centerline{\includegraphics[width=12cm]{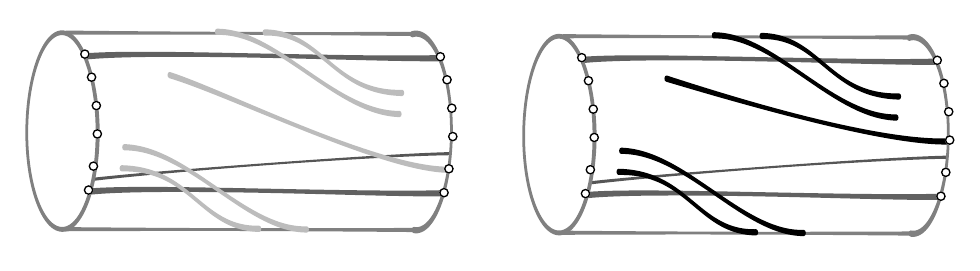}}}
\vspace{-24pt}
\end{center}
\caption{Comparing arcs of $\alpha$ and $\beta$}
\label{fig:cylinderintersections}
\end{figure}

Moreover, all of its intersections with $\alpha$ and those with $\beta$ are negative. Now, if we apply $D_{\eta_i}^{-1}$ on $a_{-}$, its intersection with $\alpha$ will increase by $n-{s+1}$ and its intersection with $\beta$ will increase by $n-s$. As such, we have
	\[
		i(a_-,\alpha\cap U_i)-i(a_-,\beta\cap U_i)\neq i(D_{\eta_i}^{-1}(a_-),\alpha\cap U_i)-i(D_{\eta_i}^{-1}(a_-),\beta\cap U_i),
	\]
where the homotopy used for counting intersection number is relative to the endpoint of the involved segments. Then, when we close up $a_-$ and $D_{\eta_i}^{-1}(a_-)$ with a same well-chosen segment, one of them will intersect $\alpha$ and $\beta$ differently. 

The last step is to close $a_-$ up to get a curve which does not backtrack along $a_-$. Notice that, on the side $\eta_i^-$, the path $a_-$ separates from $\beta$ earlier than $\alpha$. Hence we only need to consider $a_-$ going along $\alpha$ and how to connect it back. Similarly, on the side of $\eta_i^+$, we need the extension $a_+$ of $a$ to be on the left side of the extensions of both $\alpha$ and $\beta$ from $v_1$ before they separate. Notice that $\alpha$ and $\beta$ will separate eventually. When they separate, locally we can tell which one goes to the left and which one goes to the right. And we will extend $a$ following the one that goes left.

Without loss of generality, we assume that $a$ extends on both sides along $\alpha$. Notice that $a_+$ extends on the right of $\alpha$ when it leaves $\eta_i^+$ if we orient $\alpha$ with the direction pointing out from $U_i$. Consider the first time when $a_+$ intersects itself after separating from $\beta$. Then we have the following two cases, depending on whether it intersects itself from left to right or from right to left. From this point, we can extend $a_+$ to a path going back to $U_i$ on the side $\eta_i^+$.

\begin{figure}[h]
\leavevmode \SetLabels
\endSetLabels
\begin{center}
\AffixLabels{\centerline{\includegraphics[width=10cm]{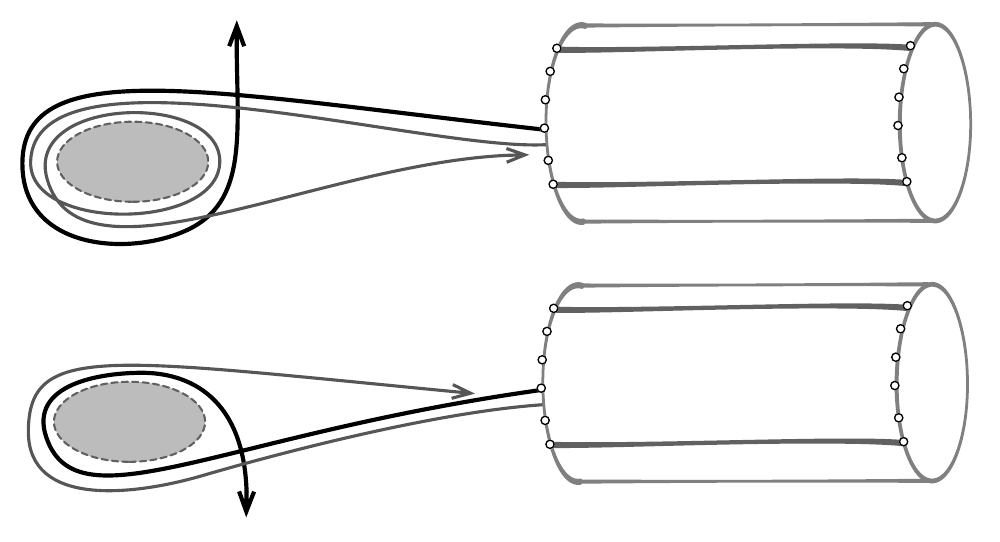}}}
\vspace{-24pt}
\end{center}
\caption{Extending the curve}
\label{fig:cylinderwind}
\end{figure}

When the $\alpha$ crosses itself from left to right, then $a_+$ follows $\alpha$ to this intersection point and then turns left to follow $\alpha$ until coming back to this intersection of $\alpha$ again, then turns right and follows itself back to $\eta^+_i$. When $\alpha$ crosses itself from right to left, then $a_+$ follows $\alpha$ to this intersection point, then turns left and goes back to $\eta_i^+$ following itself (see Figure \ref{fig:cylinderwind}).

We do the same thing on the $a_-$ side, and extend $a_-$ to a path going back to $U_i$ from the $\eta_i^-$ side. Then, connecting the two ends, we will get the curve $\gamma$. We can use the same extension for the arc $D^{-1}_{\eta_i}(a)$ and get a curve $\gamma'$. Notice that, by the construction, the cylinder $U_i$ is also $\{\gamma,\gamma'\alpha,\beta\}$-cylindrical. By the above discussion on the intersections, we have 
\begin{equation}\label{neq}
	i(\gamma,\alpha\cap U_i)-i(\gamma,\beta\cap U_i)\neq i(\gamma',\alpha\cap U_i)-i(\gamma',\beta\cap U_i).
\end{equation}

Notice that in the construction of the extension $a_+$ and $a_-$, we can add any homotopically non-trivial loop to them once they separate from $\alpha$ and $\beta$. In particular, we can construct a sequence of pairs of curves $(\gamma_n,\gamma_n')$ whose self-intersections go to infinite, such that $U_i$ is $\{\gamma_n,\gamma_n',\alpha,\beta\}$-cylindrical. Since, for each couple, inequality (\ref{neq}) holds, we may find infinitely many $k$ such that $\alpha$ and $\beta$ are not $k$-equivalent.

This concludes the last of the subcases of Case 2, and, hence, the proof of the theorem.
\end{proof}

\section{Being $k$ but not $k'$ equivalent}\label{sec:tobeornottobe}
In this section, we will construct explicit examples to show that there is no implication between $k$-equivalence and $k'$-equivalence for any pair of distinct positive integers $k$ and $k'$. 

\subsection{Constructing an example for any distinct pair $k$ and $k'$}
The rough idea is to modify a chosen curve $\gamma$ to get a pair of curves with desired property. We will first discuss the construction for the case when the starting curve $\gamma$ is simple, and then generalize the construction to the case where $\gamma$ has self-intersection. We consider $\gamma$ with a fixed orientation, and denote by $U$ one of its cylinder neighborhood. Let $U^+$ and $U^-$ be the connected components of $U\setminus\gamma$ on the left and right of $\gamma$ respectively. Let $\eta$ be a (not necessarily simple) curve intersecting $\gamma$ non-trivially. Let $x$ be one of the intersection point in $\gamma\cap\eta$. We consider the orientation on $\eta$ such that it passes $x$ from $U^+$ to $U^-$.

Now we will connect $\eta$ with $\gamma^n$ for some $n\in\mathbb{N}^\ast$ at $x$ in two different ways to get two different curves $\alpha$ and $\beta$. The precise instructions to construct them are as follows. We will denote by $y$ a point on $\eta$ different from $x$ which we will use as a base point and by $m$ some positive natural number.

\paragraph{Constructing the curve $\alpha$.} We start from $y$ and go along $\eta$ following its orientation. When we reach $x$ from the side of $U^+$ for the first time, we turn to $\gamma$ and go along it following its orientation. When we meet $x$ for the $m$-th time, we turn to $\eta$ and go along it \textit{against} its orientation until we meet $x$ from the side of $U^-$. Then we go along $\gamma$ against its orientation. When meet $x$ the $m$-th time, we turn and go along $\eta$ \textit{following} its orientation. We stop at $y$ when we meet it for the first time, and denote the resulting curve by $\alpha$.

\paragraph{Constructing the curve $\beta$.} The construction of $\beta$ is similar to that of $\alpha$. The difference is the direction to which we turn at each time when we switch between $\eta$ and $\gamma$. 

More precisely, we start from the point $y$ and go along $\eta$ following its orientation until we reach $x$ from the side of $U^+$. Then we go along $\gamma$ following its orientation. When we meet $x$ for the $m$-th time, we turn to $\eta$ and go along it \textit{following} its orientation (this is different from $\alpha$). Until we meet $x$ from the side of $U^+$, we go along $\gamma$ against its orientation. When we meet $x$ for the $m$-th time, we turn to $\eta$ and go along it following its orientation. We stop at $y$ when we meet it for the first time, and the resulting curve is denoted by $\beta$. Figure \ref{fig:twistingbig} is an illustration of the construction of $\beta$.

Figure \ref{fig:combinatorialpicture} is a more combinatorial illustration of the differences between $\alpha$ and $\beta$ in $U$. The left two figures are for $\alpha$ in $U^+$ and $U^-$ respectively, while the last one is for $\beta$ in $U$.\\

\begin{figure}[h]
\leavevmode \SetLabels
\endSetLabels
\begin{center}
\AffixLabels{\centerline{\includegraphics[width=10cm]{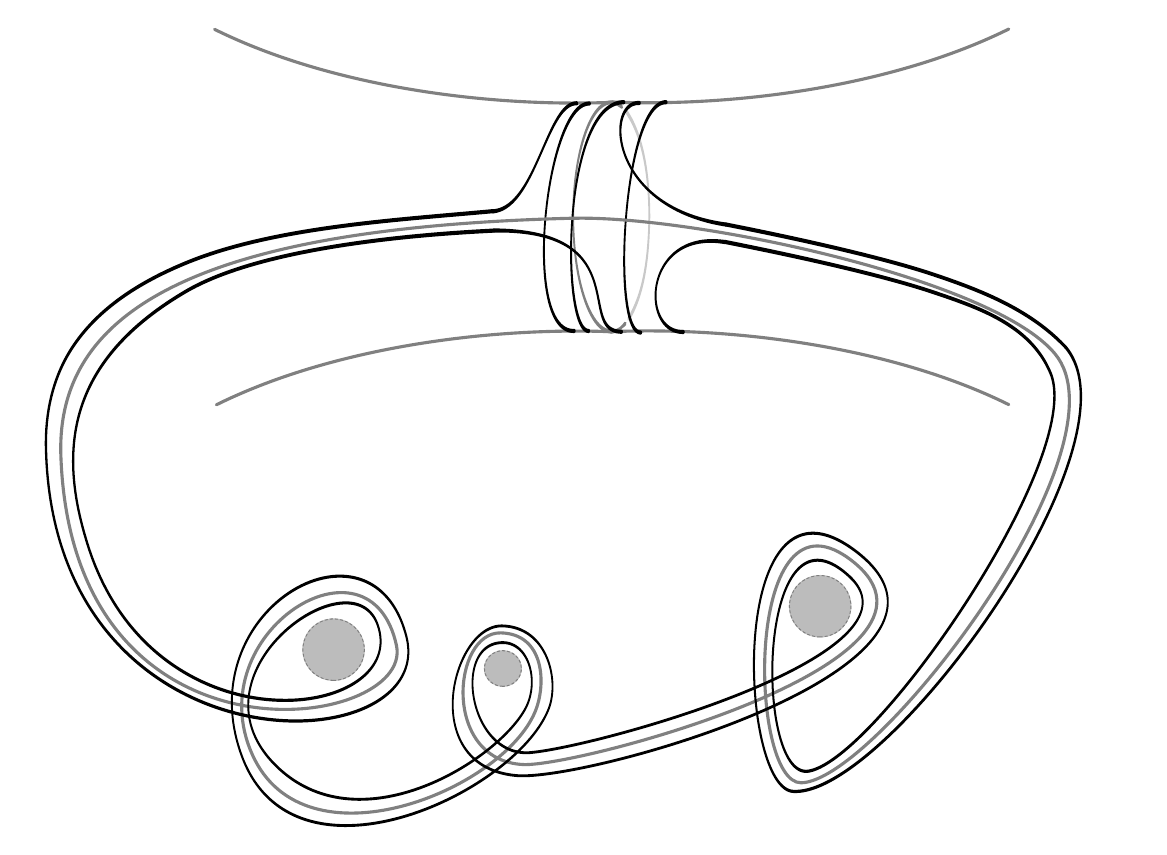}}}
\vspace{-24pt}
\end{center}
\caption{Extending the curve}
\label{fig:twistingbig}
\end{figure}

\begin{figure}[h]
\leavevmode \SetLabels
\endSetLabels
\begin{center}
\AffixLabels{\centerline{\includegraphics[width=14cm]{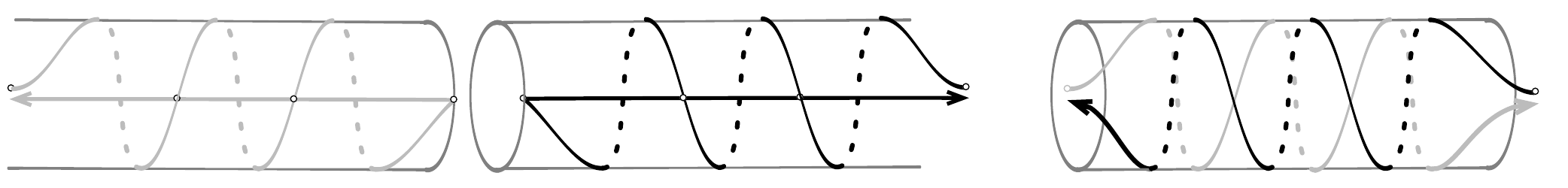}}}
\vspace{-24pt}
\end{center}
\caption{A more combinatorial picture of what is happening}
\label{fig:combinatorialpicture}
\end{figure}

Notice that in this construction, we do not use the fact that $\gamma$ is simple. In fact, the only thing that we use is that the surface is oriented and, once we give a curve an orientation, locally we can talk about left and right. Hence the construction works for arbitrary $\gamma$. The only thing which we have to be careful about is that the curve $\eta$ should not have segments going along $\gamma$ already (in other words, we assume that both $\alpha$ and $\beta$ are taut).

Now we need to verify that the curves $\alpha$ and $\beta$, constructed from well-chosen $\gamma$ and $\eta$, satisfy the desired property:
\begin{lemma}\label{lem:techlem}
	Let $\gamma$ be a $k$-curve with $k>0$. Let $m>0$ be an integer. The curves $\alpha$ and $\beta$ constructed above
	\begin{enumerate}
		\item intersect $\gamma$ a different number of times and hence are not $k$-equivalent;
		\item are $k'$-equivalent for any $k'<m^2k$ different from $k$.
	\end{enumerate}
\end{lemma}
\begin{proof}
	To prove the first point, let us first consider the case when $\gamma$ is simple. Using the bigon criterion, from the construction of $\alpha$ and $\beta$, we have the identity
		\[
			i(\alpha,\gamma)+2=i(\beta,\gamma).
		\]
	In particular, they are not $0$-equivalent. Now we observe that the above equality also holds when $\gamma$ is not simple. This follows from the fact that the local pictures are the same, and the curves are in minimal position again by the bigon criterion. This proves the first statement.

	Now we pass to the second. To visualize the construction, we consider the case when $\gamma$ is simple and oriented and although the construction will not work in this case, it is helpful to understand what is happening.
	
	Let $U$, $U^+$ and $U^-$ be as in the construction of $\alpha$ and $\beta$. We denote the boundary components of $U^+$ and $U^-$ that are different from $\gamma$ by $\gamma^+$ and $\gamma^-$ respectively. We call a strand an arc entering and leaving $U$ from the same boundary component. Up to a homotopy relative to $\partial U$, such a strand is disjoint from $\gamma$ and is contained in either $U^+$ or $U^-$. Note that a strand is not necessarily simple. 
	
As before, we associate to $\gamma$ an orientation. Let $a$ be an oriented strand in $U^+$. By connecting the end points of $a$ with a segment of $\gamma^+$, we obtain a curve which is homotopic to $\gamma^l$ with $l\in\ZZ$. Notice that there are two ways to close up $a$ with segments on $\gamma^+$ and the resulting closed curves homotopic to $\gamma^{l_1}$ and $\gamma^{l_2}$ with $|l_1-l_2|=1$. We define
		\[
			\mathrm{W}(a)=\max\{|l_1|,|l_2|\}.
		\]
		
	Let $a$ and $b$ be two oriented strands contained both in $U^+$, and we are interested in the intersection number between $a$ and $b$ in $U^+$. Here the homotopy is considered relative to $\gamma^+$. By the bigon criteria, given any strand $a$, if $b$ and $b'$ are two strands with the same end points, such that up to homotopy $b$ is part of $b'$ as subset and
		\[
			W(a)\le W(b),
		\]
	then
		\[
			2(W(a)-1)\le i(a,b)=i(a,b')\le 2W(a).
		\]
	
	Let us use the same notation, and let $U$ be the cylinder neighborhood of $\gamma$ used in the construction of $\alpha$ and $\beta$. By construction, $\alpha$ is obtained from two copies of $\eta$ by connecting them by strands $a^+$ and $a^-$ at $x$ in $U$, while $\beta$ is obtained from two copies of $\eta$ by connecting them by arcs $b^+$ and $b^-$ at $x$ in $U$. By the choice made in the construction, up to homotopy, as subsets those strands in $\alpha$ can be viewed as part of the arcs in $\beta$. Therefore, for any strand $c$ in $U$ intersecting $a^+$ with $W(c)<W(a^+)=m+1$, we have
		\[
			i(c,a^+)=i(c,b^+).
		\]
	
For any arc in $U$, its intersections with $a^+$ and with $a^{-}$ are both $m$. On the other hand, its intersection with $b^+\cup b^-$ is greater or equal to $2m$. Notice that $b^+$ and $b^-$ can be viewed as applying a left Dehn twist $D_\gamma$ and right Dehn twist $D_\gamma$ on an arc in $\eta\cap U$, intersecting $\gamma$ at $x$, $m$ times respectively. If an arc $U$ is obtained by applying $D_\gamma$ $s$ times on the chosen arc in $\eta$, then, by Lemma \ref{lem:arcs}, its intersection number with $b^+\cup b^-$ is
		\[
			|s+m|+|s-m|.
		\]
Hence, when $|s|\le m$, this quantity is $2m$ and otherwise it will be greater than $2m$.
	
From the above discussion, we can see that, to have different intersections with $\alpha$ and $\beta$, the curve should be either $\gamma$, or have a strand in $U$ with $W$-value greater than $m$ or have an arc with $s$, as defined above, greater than $m$. In either of the last two cases, the curve has to go along $\gamma$ for $m'$ times with $m'>m$.
	
Now we come back to the case when $\gamma$ is a $k$-curve with $k>0$. Then following the above discussion, in order to have different intersections with $\alpha$ and $\beta$, the curve is either $\gamma$ or goes along $\gamma$ for at least $m+1$ times. By bilinearity of the intersection number $i$, we have
		\[
			i(m\gamma,m\gamma)=m^2i(\gamma,\gamma)=2m^2k.
		\]
which is a lower bound for the self-intersection of a curve different from $\gamma$ and intersecting $\alpha$ and $\beta$ differently. Hence $\alpha$ and $\beta$ cannot be distinguished by any $k'$-curve different from $\gamma$ with $k'<m^2k$. 
\end{proof}

\subsection{From integers to sets of integers}

Using Lemma \ref{lem:techlem}, we can prove the following more general result.
\begin{proposition}
Let $K$ and $K'$ be two disjoint finite collections of positive non-zero integers. There exists a pair of curves $\alpha$ and $\beta$ which are not $k$-equivalent for any $k\in K$ and $k'$-equivalent for any $k'\in K'$.
\end{proposition}

\begin{proof}
	Let $K=\{k_1,\dots,k_s\}$ and $K'=\{k_1',\dots,k_t'\}$. Let $k_{\textrm{max}}'=\max K'$. Let $m$ be a positive integer such that $m^2k_i>k_{\textrm{max}}'$ for all $i\in\{1,...,s\}$. For each $i\in\{1,...,s\}$, let $\gamma_i$ be a $k_i$-curve. We choose $\eta$ to be a curve which intersects all the curves $\gamma_i$ non trivially and has no segments going along any of $\gamma_i$s. We apply the construction in the first part of this section to get a pair of desired curves, by doing it several times. 
	
	For our convenience, we call the construction of $\alpha$ the construction I, and that of $\beta$ the construction II. We construct $\alpha_1$ from adjoining $\eta$ with $2m$ copies of $\gamma_1$ by the construction I. If we have constructed $\alpha_i$, then we construct $\alpha_{i+1}$ from adjoining $\alpha_i$ with $2m$ copies of $\gamma_{i+1}$ by the construction I. In this way, we get a curve $\alpha_s$. Similarly, if in each step we use construction II, then we get another sequence of curves and denote the last one by $\beta_s$.
	
	A similar discussion as in the proof of Lemma \ref{lem:techlem} will show that the two curves $\alpha_s$ and $\beta_s$ are not $k$-equivalent for any $k\in K$, but $k'$-equivalent for any $k'\in K'$.

\end{proof}
\begin{remark}
Observe that, so far, we do not require any restrictions on the genus of $\Sigma$ and so the results all hold for any surface of genus $g\ge2$.
\end{remark}

\section{How many curves do we need to know if two curves are the same?}\label{sec:infinite}

For two distinct curves $\alpha$ and $\beta$, the previous discussion shows that they can be $k$-equivalent for any finite number of integers $k$. Thus in order to be able to distinguish between two arbitrary curves using their intersections with a general reference system, the reference system must have $k$-curves for all $k$. If however we know self-intersection numbers of $\alpha$ and $\beta$, we will show that a reference system can be chosen which only consists only curves with an explicit upper bound on their self-intersection. 

We will consider two cases: when $\alpha$ and $\beta$ are simple and when $\alpha$ and $\beta$ have self-intersections smaller than a constant $L>0$. 

We begin with the case where the curves are simple in order to illustrate the approach.

\subsection{Finite simple length rigidity}
The following is well-known (see for instance \cite{Hamenstadt}). 
\begin{proposition}
There exists a finite collection $\C$ of simple curves such that any pair of simple curves $\alpha$ and $\beta$ are homotopic to each other if and only if they have the same number of intersection with every curve in $\C$.
\end{proposition}
\begin{proof}
We sketch a quick proof to highlight the difference between simple and non-simple curves. 

The statement essentially boils down to using Dehn-Thurston coordinates. In order to define Dehn-Thurston coordinates, we need a marked pants-cylinder decomposition of the surface. 
		
We say that a curve is \textit{almost embedded} in a graph if the curve cover all edges of the graph at most twice and if there is an edge of the graph which is covered only once. The curves on the boundary of a contractible neighborhood of the dual graph are called the dual curves to the pants decomposition.
	
Roughly speaking, the intersection coordinates can be read by considering the intersection number with the pants curves. We would like to recover the twist coordinates by checking the intersection number with the dual curves. However, only taking pants curves and their dual curves are not enough, and there are several difficulties to overcome.

The first one is that the intersection may not be essential: when we try to push curves into the neighborhood of the dual graph as much as possible, the curve that we consider and the dual curves may not be in minimal position. The second one is that even if we assume that these intersections are essential, we have to find a way to read the sign of the twist coordinates by looking at the intersection with the dual curves. The third one is that intersection with each dual curve includes all twist coordinates related to this curve. We do not get each twist coordinates separately, but rather a sum.
	
The solution is to add the image of dual curves under Dehn twists along pants curves into the reference system. The collection that we will take consists of curves in a pants decomposition, the collection of simple curves which are almost embedded in a dual graph of a pants decomposition, and their images under the left Dehn twist along the pants curves.
\end{proof}
\begin{remark}
One way to view this result is a version of finite length rigidity for simple curves (and more generally for measured laminations). The existence of $0$-equivalent curves suggests that this is no longer true for curves with self-intersections as we shall see in the sequel.
\end{remark}

\subsection{Bounded intersections}
Assume now that $\alpha$ and $\beta$ have at most $L$ self-intersection points for some $L>0$. What follows is similar to our previous proofs. 

\begin{proposition}
There exists a universal constant $C>0$ such that intersection with $k$-curves for $k\leq CL$ distinguishes between $\alpha$ and $\beta$.
\end{proposition}

\begin{remark}
We have no idea what the optimal constant $C$ is, but we do know it can be taken to be $17$. 
\end{remark}
\begin{proof}
If $\alpha$ and $\beta$ are not $0$-equivalent, then by definition $\alpha$ and $\beta$ can be distinguished by their intersections with $0$-curves.	

Now we assume that $\alpha$ and $\beta$ are $0$-equivalent. We improve the construction of the curve $\gamma$ constructed in the case 2.2. in the proof of Theorem \ref{thm:otal} as following: instead of choosing the arc $a$ by connecting $(v,v_1)$ and $(w_1,w_2)$ which is contained in $\Sigma^+$, we consider the arc $b$ which connects $(v,v_1)$ and $(w_{n-1},w_n)$ from going through $\Sigma^-$. The rest of the construction would be similar. We denote by $q^+$ (resp. $q^-$) the extension of $b$ on the side $\partial_i^+$ (resp. $\partial_i^-$).
	
We denote by $\gamma$ the curve that we constructed in this way. Then its self-intersection $i(\gamma,\gamma)$ can be written as a sum $i(q^+,q^+)+i(q^+,q^-)+i(q^-,q^-)+2i(q^+,b)+2i(q^-,b)$. By their definition, the segments $q^+$ and $q^-$ are at most two copies of $\alpha$ or $\beta$. We first consider the following lemma about $i(\alpha,\beta)$.

If $i(\alpha,\beta)>2L$, then since we have
	\[
		i(\alpha,\alpha)\le 2L\textrm{ and }i(\beta,\beta)\le 2L,
	\]
the curve $\alpha$ and $\beta$ are not $k$-equivalent where $i(\alpha,\alpha)=2k$ or $i(\beta,\beta)=2k$. Then let $C=1$ and we obtained the proposition.

By the lemma above, we only have to consider the case where $i(\alpha,\beta)\le 2L$. Since, by our construction, $q^+$ and $q^-$ follow either part of $\alpha$ or part of $\beta$ at most twice, we have
	\[
		i(q^+,q^+)+i(q^+,q^-)+i(q^-,q^-)\le3\max\{4i(\alpha,\alpha),4i(\beta,\beta),4i(\alpha,\beta)\}\le24L.
	\]
On the other hand, by construction, both of them intersect $b$ at most once. Hence $i(\gamma,\gamma)<28L$.

Since we only know that one of the curves $\gamma$ and $D_{\eta_i}(\gamma)$ can distinguish $\alpha$ and $\beta$ from each other, we still have to bound the self-intersection number of $D_{\eta_i}(\gamma)$. To do this, we have to estimate $i(\eta_i,\alpha)$ and $i(\eta_i,\beta)$. Notice that this number can be arbitrarily large and so we will need to modify the pants decomposition that we first considered, and find a new pants decomposition such that we can bound the intersection of $\alpha$ and $\beta$ with the new pants curves.
	
The strategy that we would like to apply is to use a formula due to Dylan Thurston \cite{Thurston} about the intersection of closed curves with simple curves. Let $\gamma$ be a non-simple curve. Each of its self-intersection points can be resolved in two possible ways (see Figure \ref{fig:resolve}). 

\begin{figure}[h]
\leavevmode \SetLabels
\endSetLabels
\begin{center}
\AffixLabels{\centerline{\includegraphics[width=6cm]{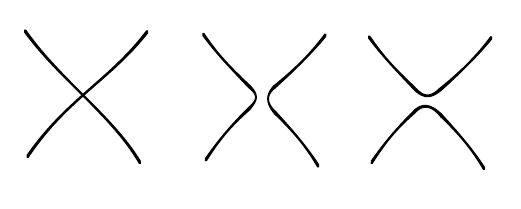}}}
\vspace{-24pt}
\end{center}
\caption{The two ways to resolve an intersection}
\label{fig:resolve}
\end{figure}

Let $\gamma_1$ and $\gamma_2$ be the two multicurves (with either one or two components) that are the result of this resolution. Now let $\eta$ be a simple curve. The formula is:
\begin{equation}\label{eqn:thurston}
i(\gamma,\eta)=\max\{i(\gamma_1,\eta),i(\gamma_2,\eta)\}
\end{equation}
 
We use the formula as follows. We begin by resolving $\alpha$ and $\beta$ into multicurves until each of their components is simple and disjoint. We then apply the mapping class group to get a new multicurve which intersects the pants curves as few times small as possible. Then we glue them back and denote the resulting curve by $\alpha'$ and $\beta'$. We would like to show that $\alpha'$ and $\beta'$ are the image of $\alpha$ and $\beta$ under a same mapping class group element action, and their intersections with pants curves can be controlled by $L$.
	
By our construction, all self-intersections of $M=\{\alpha,\beta\}$ are between arcs in cylinders. Hence we can resolve each intersection so that the parallel segments that we get from the cross are along the height of the cylinder (otherwise the intersection with pants curves will decrease). By resolving all self-intersections of $M$ in this way, we get a multicurve $C_M$ which has the same intersection number with each pants curve.

We consider the topological type of the multicurve which is a collection curves which is pairwise disjoint and pairwise non-homotopic to each other. The topological type of a multicurve is determined by the topological type of its complement, and up to mapping class group action, there are only finitely many of them.
	
Hence there exists a multicurve $C_M'$ of the same topological type which intersects the pants curves and their dual curves the least amount of times. As such there is a mapping class group $\phi$ sending $C_M$ to $C_M'$, which also sends $M$ to $M'=\{\alpha',\beta'\}$. We apply the map $\phi^{-1}$ to the current pants-cylinder decomposition to get a new one. From now on, we consider this new one We denote by $\eta_i'$ the curves in the new pants decomposition.
	
Since
	\[
		i(M,M)=i(\alpha,\alpha)+i(\alpha,\beta)+i(\beta,\beta)\le6L,
	\]
there are at most $6L$ curves. Hence the intersection numbers $i(\alpha,\eta_i')$ and $i(\beta,\eta_i')$ are bounded by $6L$.
	
Combining this with our previous result, we can conclude that we can find a curve $\gamma$ with $i(\gamma,\gamma)\le34K$ such that $i(\gamma,\alpha)\neq i(\gamma,\beta)$. Hence we can choose $C=17$, and the proposition is proved.
\end{proof}
	
\subsection{Infinitely many curves are always needed}
Here we show Corollary \ref{cor:infinite} which implies that in order to distinguish between non-simple curves, even if you have prior knowledge that they have the same self-intersection, the intersection with infinitely many curves is required.

Let $k>0$. Let $\eta_1, \hdots, \eta_n$ be a reference system of curves. We need to show that the map $\varphi:\C_k \to \N$ given by

$$\varphi(\alpha)= \left(i(\alpha,\eta_1), \hdots, i(\alpha,\eta_n)\right)
$$
is not injective. To do so we need to construct two distinct curves, say $\gamma$ and $\gamma'$, which have the same intersection with each curve in the finite reference system.

We start by considering a $\{\eta_1,\dots,\eta_n\}$-cylindrical pants decomposition $P$ and a curve $\zeta$ dual to $P$. Let $U$ and $U'$ be two cylinder neighborhoods associated to two pants curves $\eta$ and $\eta'$ through which $\zeta$ passes. Now we take another copy of $\zeta$, and twist it in $U$ and $U'$ such that it has twist parameters $1$ and $-1$ in $U$ and $U'$ (in Dehn-Thurston coordinates). This new curve we call $\zeta'$.

First observe that $\zeta$ and $\zeta'$ are disjoint and that their union in $\Sigma\setminus(U\cup U')$ are two pairs of parallel paths. We have two ways to connect them into a figure-8 curve by adding a crossing in one of the two pairs of parallel paths. Notice that these two curves are homotopically different as the part connecting the two pairs of parallel paths has non-trivial topology. We denote the two curves by $\gamma_0$ and $\gamma_0'$.
\begin{figure}[h]
\leavevmode \SetLabels
\endSetLabels
\begin{center}
\AffixLabels{\centerline{\includegraphics[width=12cm]{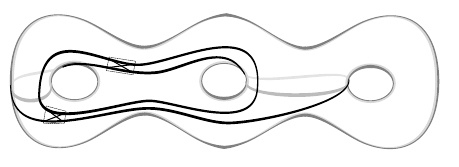}}}
\vspace{-24pt}
\end{center}
\caption{The curves $\gamma_0$ and $\gamma_0'$}
\label{fig:genus3curvejoin}
\end{figure}

Now in order to get $k$ self-intersections, instead of $\zeta$, we consider $\zeta^k$, let $\zeta'$ be the same curve as above, and repeat the above construction to get two curves $\gamma_1$ and $\gamma_1'$. Notice that $P$ is the $\{\eta_1,\dots,\eta_n\}$-cylindrical pants decomposition, hence all $\eta_i$s have twist parameter either $0$ or $1$. We apply Dehn twists at $U$ and $U'$ to increase the twisting parameters of $\gamma_1$ and $\gamma_1'$, and such that their arcs in $U$ and $U'$ intersect all $\eta_i$s with negative twisting parameters. The resulting curves are denoted by $\gamma$ and $\gamma'$. We denote by $\phi$ the mapping class used here.

We now claim that they intersect each $\eta_i$ the same number of times. The reason is that, by construction, all curves only intersects in the cylinders, and the curves $\phi(\zeta^k)$ and $\phi(\zeta')$ intersect all the $\eta_i$s positively. We claim that this implies that they are minimal position as this implies the bigon criterion: if there was a bigon, there would be intersections with different signs side by side. 

The set $\phi(\zeta^k\cup\zeta')\cap U$ is a union of arcs intersecting $\eta_i$ negatively. Adding extra crossing between arcs in $U$ will not change the fact that they are in minimal intersecting position with $\eta_i$ and their intersection number with $\eta_i$. Hence we found distinct $\gamma$ and $\gamma'$ such that their images under $\varphi$ are the same, which proves Corollary \ref{cor:infinite}.

\appendix

\section{Measured hexagon decompositions of bordered surfaces}\label{app:hexagon}
This appendix is used in Section \ref{IIC}. Let $\Sigma$ be an oriented surface of genus $g>0$ with $n>0$ boundary components. We denote by $\partial\Sigma$ the boundary of $\Sigma$. A \textit{hexagon decomposition} $T$ of $\Sigma$ is a maximal collection of disjoint pairwise non-homotopic simple arcs of $\Sigma$ with endpoints on $\partial \Sigma$. By cutting along arcs in $T$, the surface $\Sigma$ is decomposed into hexagons, hence the name. A \textit{measured hexagon decomposition} is a couple $(T,f)$, where $T$ is a hexagon decomposition of $\Sigma$ and $f$ is map from $T$ to $\R$ (note there is no condition on the measures being positive). 

Let $T=\{\alpha_1,...,\alpha_n\}$ be a hexagon decomposition and $(T,f)$ be a measured hexagon decomposition. Let $\gamma$ be a curve on $\Sigma$. The quantity
	\[
		l_f(\gamma)=\sum_{i=1}^nf(\alpha_i)i(\alpha_i,\gamma),
	\]
is called the \textit{$f$-length} of $\gamma$.

\begin{proposition}[Finite length rigidity]
Let $T$ be an hexagon decomposition of $\Sigma$. There exist $s\in\mathbb{N}^\ast$ and finitely many curves $\gamma_1,\dots,\gamma_s$ in $\Sigma$, such that for any pair of measures $f_1$ and $f_2$, we have $f_1=f_2$ if and only if $l_{f_1}(\gamma_i)=l_{f_2}(\gamma_i)$ for any $1\le i\le s$.
\end{proposition}
\begin{proof}
We consider the dual graph $\Gamma$ of $T$. We denote by $e_i$ the edge of $\Gamma$ intersecting $a_i$.
	
Notice that $\Sigma$ admits a contraction to $\Gamma$. Each curve in $\Sigma$ has a representative in $\Gamma$. Moreover, the measure $f$ on $T$ induces a measure on the edge set of $\Gamma$, still denoted by $f$. The $f$-length of each path in $\Gamma$ is given by adding the $f$-lengths of all edges in the path together.
	
Since $\Gamma$ is a dual graph of a hexagon decomposition, it is trivalent. Given any edge $e_j$, it is adjacent to $4$  half-edges, denoted by $a_1$, $a_2$, $a_3$ and $a_4$. Since there are no valence 1 vertices in $\Gamma$ and $\Gamma$ is connected, the four half-edges are paired such that  half-edges in each pair are connected by a paths in $\Gamma\setminus e_i$. The two paths together, with $e_j$ form a graph (which may not be embedded). Depending whether the paired half-edges are adjacent to a same vertex or to two different vertices, the graph is either a theta type graph or a dumbbell type graph (see Figure \ref{fig:appendix}).
	
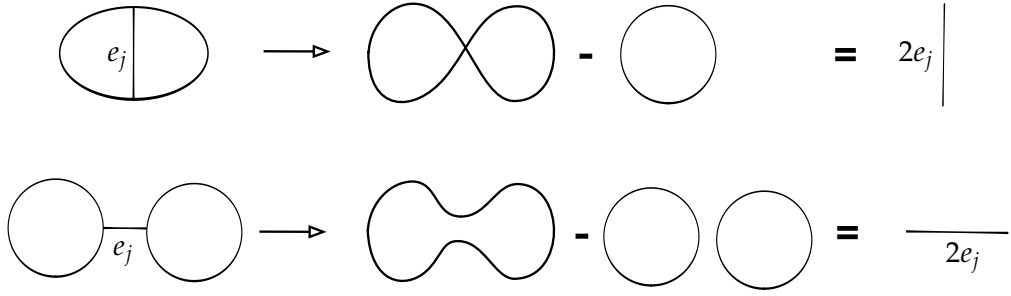
\begin{figure}[h]
\leavevmode \SetLabels
\L(.155*.73) $e_j$\\%
\L(.16*.16) $e_j$\\%
\L(.83*.73) $2 e_j$\\%
\L(.87*.14) $2 e_j$\\%
\endSetLabels
\begin{center}
\AffixLabels{\centerline{\includegraphics[width=14cm]{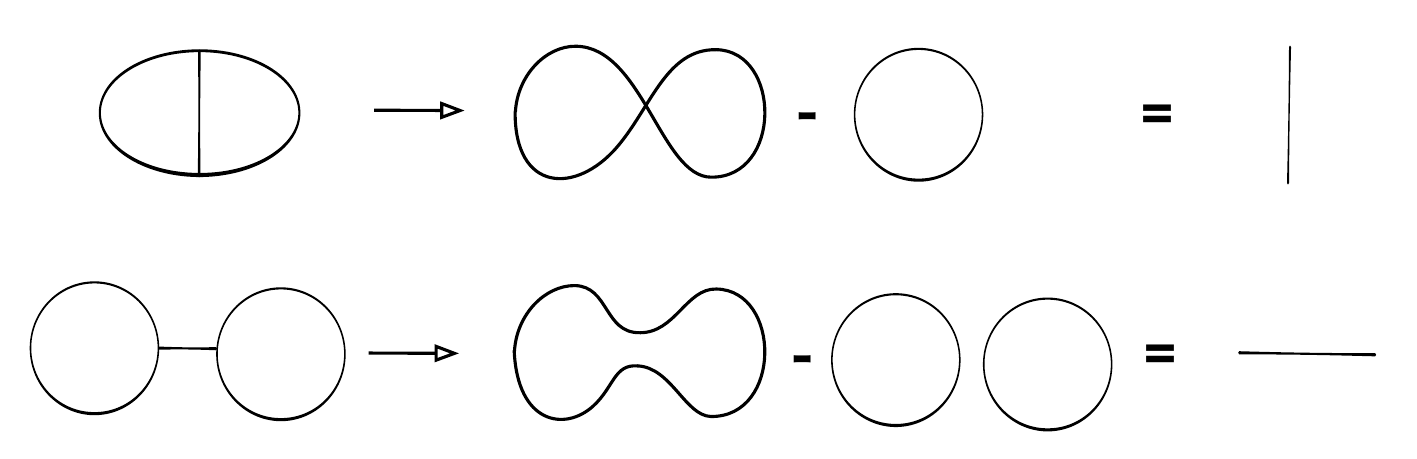}}}
\vspace{-24pt}
\end{center}
\caption{The theta and dumbbell graphs}
\label{fig:appendix}
\end{figure}

Hence the $l_f(e_j)$ is determined by the $f$ lengths of two loops $\gamma_j$ and $\gamma_j'$. Therefore the $f$-lengths of curves in the collection
		\[
			\{\gamma_1,\gamma_1',...,\gamma_n,\gamma_n'\},
		\]
determine $f$.
\end{proof}

We can moreover improve this result by requiring all the curves $\gamma_i$ in the statement to be simple.

\begin{proposition}[Finite simple length rigidity]\label{finitesimplelengthrigidity}
Let $T$ be an hexagon decomposition of $\Sigma^-$. There exist finitely many simple curves $\gamma_1',\dots,\gamma_s'$ in $\Sigma^-$ such that, for any pair of measures $f_1$ and $f_2$, $f_1=f_2$ if and only if $l_1(\gamma_i')=l_2(\gamma_i')$, where $l_j(\gamma_i')$ is the $f_j$-length of $\gamma_i'$.
\end{proposition}

This proposition follows by applying the following lemma to the formula of the $f$-length of a curve.

\begin{lemma}
Let $T$ be a hexagon decomposition of $\Sigma$, and $\gamma$ be any non simple curve in $\Sigma$. Let $\gamma_+$ and $\gamma_-$ be the curves of $\Sigma$ obtained from $\gamma$ by resolving one self-intersection point. Then
	\[
	i(T,\gamma)=\max\{i(T,\gamma_+),i(T,\gamma_-)\}.
	\]
\end{lemma}

\begin{proof}
The proof follows from Equation \ref{eqn:thurston} (see \cite{Thurston}) by doubling the surface. 
\end{proof}

{\it Addresses:}\\
Department of Mathematics, FSTM, University of Luxembourg, Esch-sur-Alzette, Luxembourg\\
School of Mathematical Sciences, Nankai University, Tianjin, China\\
{\it Emails:}\\
hugo.parlier@uni.lu\\
binbin.xu@nankai.edu.cn

\end{document}